\documentclass[12pt,reqno]{amsart}

\usepackage[utf8]{inputenc}
\usepackage{lmodern}
\usepackage{framed} 
\usepackage{color}
\usepackage{xcolor}
\usepackage{bbm}
\usepackage{mathrsfs}
\usepackage{enumitem}
\usepackage{amsmath}
\usepackage{amssymb}
\usepackage{amsfonts}
\usepackage{latexsym}
\usepackage{amsthm}
\usepackage[english]{babel}
\usepackage[T1]{fontenc}
\usepackage{caption}
\usepackage{pgf,tikz}
\usetikzlibrary{arrows}
\usetikzlibrary{shapes}
\usetikzlibrary{patterns}
\usetikzlibrary{shapes.geometric}
\usepackage{geometry}
\usepackage{graphicx}
\usepackage{fancyhdr}
\usepackage{faktor}
\usepackage{mdframed}
\usepackage{ytableau}
\usepackage{cellspace}
\usetikzlibrary{matrix,automata,topaths,decorations.pathreplacing}

\definecolor{Theme}{gray}{0}

\makeatletter

\renewenvironment{proof}[1][\proofname]{\par
   \pushQED{\begin{center}\textcolor{Theme}{\ensuremath{\blacksquare}}\end{center}}%
  \normalfont \topsep6\p@\@plus6\p@\relax
  \trivlist
  \item[\hskip\labelsep
        \bfseries
    #1\@addpunct{.}]\ignorespaces
}{%
  \popQED\endtrivlist\@endpefalse
}

\newtheoremstyle{Thm}  
  {10pt}   
  {20pt}   
  {\itshape}  
  {0pt}       
  {\bfseries} 
  {}         
  {\newline}  
  {\textcolor{Theme}{\thmname{ #1}\thmnumber{ #2}}\textbf{\thmnote{ (#3)}}} 
  
  \newtheoremstyle{Thmused}  
  {10pt}   
  {20pt}   
  {\itshape}  
  {0pt}       
  {\bfseries} 
  {}         
  {\newline}  
  {\textcolor{Theme}{\thmname{ #1}} \textbf{\thmnote{#3}}} 
  
\newtheoremstyle{Def}  
  {10pt}   
  {20pt}   
  {}  
  {0pt}       
  {\bfseries} 
  {}         
  {\newline}  
  {\textcolor{Theme}{\thmname{ #1} \thmnumber{#2}}\textbf{\thmnote{ (#3)}}} 

\theoremstyle{Thmused} 
\theoremstyle{Thm} \newtheorem{Thm}{Theorem}
\theoremstyle{Thm} \newtheorem{Lemma}{Lemma}[section]
\theoremstyle{Thm} \newtheorem{Prop}{Proposition}[section]
\theoremstyle{Thm} \newtheorem{Cor}{Corollary}[Prop]
\theoremstyle{Thm} 
\theoremstyle{Thm} 
\theoremstyle{Def} 
\theoremstyle{Def} 
\theoremstyle{Def} 
\theoremstyle{Def} \newtheorem*{Rem}{Remark}
\theoremstyle{Def} 

\let\<\langle
\let\>\rangle
\newcommand{\R}{\mathbbm{R}}
\newcommand{\K}{\mathbbm{K}}
\newcommand{\C}{\mathbbm{C}}

\newcommand{\Oo}{\mathbbm{O}}
\newcommand{\Z}{\mathbbm{Z}}
\newcommand{\Hh}{\mathbbm{H}}
\newcommand{\D}{\mathbbm{D}}

\newcommand{\wt}{\widetilde}

\newcommand{\SL}{\mathop{\rm{SL}}}

\newcommand{\SO}{\mathop{\rm{SO}}}
\newcommand{\U}{\mathrm{U}}
\newcommand{\Super}{\mathrm{S}}
\newcommand{\Ss}{\mathrm{S}}
\newcommand{\SU}{\mathop{\rm{SU}}}
\newcommand{\Sp}{\mathop{\rm{Sp}}}
\newcommand{\Spin}{\mathop{\rm{Spin}}}

\newcommand{\F}{\mathrm{F}}
\newcommand{\SOe}{\mathop{\rm{SO_e}}}

\newcommand{\Mat}{\mathop{\rm{Mat}}}
\newcommand{\Res}{\mathop{\rm{Res}}}

\newcommand{\Tr}{\mathop{\rm{Tr}}}
\newcommand{\End}{\mathop{\rm{End}}}
\newcommand{\Id}{\mathop{\rm{Id}}}
\newcommand{\Ind}{\mathop{\rm{Ind}}}

\newcommand{\triv}{\mathrm{triv}}

\newcommand{\p}{\mathfrak{p}}
\newcommand{\g}{\mathfrak{g}}
\newcommand{\h}{\mathfrak{h}}
\newcommand{\m}{\mathfrak{m}}
\newcommand{\n}{\mathfrak{n}}
\newcommand{\Ak}{\mathfrak{k}}
\newcommand{\Aa}{\mathfrak{a}}

\newcommand{\At}{\mathfrak{t}}

\newcommand{\soe}{\mathfrak{so}}

\newcommand{\HC}{\mathscr{H}\hspace{-0.4mm}\mathscr{C}}
\newcommand{\Zentrum}{\mathcal{Z}}
\newcommand{\E}{\mathscr{E}}
\newcommand{\RR}{\mathscr{R}}
\newcommand{\Ad}{\mathsf{Ad}}

\newcommand{\Hi}{\mathscr{H}}

\newcommand{\IWk}{{\textrm{\bf k}}}
\newcommand{\IWH}{{\textrm{\bf H}}}
\newcommand{\IWn}{{\textrm{\bf n}}}

\newcommand{\fonction}[5]{\begin{array}{lrcl}
#1: & #2 & \longrightarrow & #3 \\
    & #4 & \longmapsto & #5 \end{array}}

\renewcommand{\textbf}[1]{\begingroup\bfseries\mathversion{bold}#1\endgroup}

\newcommand*\circled[1]{\tikz[baseline=(char.base)]{
            \node[shape=circle,draw,inner sep=2pt] (char) {#1};}}

\makeatother

\definecolor{Rouge}{RGB}{200,0,0}
\definecolor{Vert}{RGB}{0,147,83}
\definecolor{Bleu}{RGB}{0,89,157}

\title[Residue representations - the rank one case]{Residue representations - the rank one case}
\author[Resonances on homogeneous vector bundles]{Simon ROBY}
\address{Yau Mathematical Sciences Center, Tsinghua University, Beijing, 100084, China}
\email{roby@mail.tsinghua.edu.cn}
\date{}
\subjclass[2010]{Primary: 22E45, 20G05, 22D10 ; secondary: 43A85,58J50}

\begin{document}

\maketitle

\tableofcontents

\begin{abstract}
With each resonance of the Laplacian acting on the compactly supported sections of a homogeneous vector bundle over a Riemannian symmetric space of the non-compact type, One can associate a \emph{residue representation}. The purpose of this paper is to study them. The symmetric space is assumed to have rank-one but the irreducible representation $\tau$ of $K$ defining the vector bundle is arbitrary. We give an algorithm that aims at determining if these representations are irreducible, finding their Langlands parameters, their Gelfand-Kirillov dimensions and wave front sets. As an example, we apply this algorithm to the Laplacian of the $p$-forms in the cases of all the classical real rank-one Lie groups. 
\end{abstract}

\section{Introduction}

Let $G$ be a connected non-compact real semisimple Lie group with finite center, $K$ a maximal compact subgroup of $G$ and $G/K$ the corresponding Riemannian symmetric space of non-compact type. Let $(\tau,V_\tau)$ be an representation of $K$, which we will assume without loss of generality to be irreducible. We consider the homogeneous vector bundle $E_\tau = (G\times V_\tau )/\sim$ over $G/K$, where $(g,v)\sim(gk^{-1},\tau(k)\cdot v)$ for all $g\in G$, $k\in K$ and $v\in V_\tau$. If $\tau$ is the trivial representation of $K$, then the bundle $E_\tau$ can be identified to $G/K$. We will refer to this case as the \emph{scalar case}, in opposite to the \emph{bundle case}, when $\tau$ is not trivial. 
If the Lie group $G$ is the Lorentz group $\SOe(n,1)$ and $K = \SO(n)$, we get $G/K = H^n(\R)$, the real hyperbolic space. If $V_\tau$ is one dimensional, $E_\tau$ is a line bundle over $G/K$. 
The symmetric space $G/K$ has maximal flat subspaces, all of the same dimension, called the (real) rank of $G/K$. For instance, the rank of $H^{n}(\R)$ is 1. 
Since $G/K$ is a symmetric space of the Lie group $G$, all natural operators acting on $G/K$, like the Laplacian and its resolvent, are $G$-invariant. They can therefore be studied using the representation theory of $G$. We can generalize these operators to the bundle case, using the Casimir operator of $U(\g_\C)$ (see sections \ref{subsection hom diff op}, \ref{subsection Laplace}). We get operators acting on the sections of $E_\tau$. One can therefore address the problem of the meromorphic continuation of the resolvent of the Laplacian across its spectrum when restricted to smooth functions with compact support (a smooth sections with compact support in the bundle case ). The poles of meromorphically extended resolvent, if any, are called the resonances of the Laplacian.
The study of these poles in the scalar case has been carried out by several authors. Among them, we mention Guillop\'{e} and Zworski \cite{GuillopeZworski95}, Miatello and Will \cite{MiaWill}, and Hilgert and Pasquale \cite{HilgPasq} for the rank one case. The scalar higher-rank case is a longstanding open problem. Partial results were obtained by Mazzeo and Vasy \cite{MaVa05} and Strohmaier \cite{Stro05}. Complete results for most of the rank-two cases were proved in a series of papers by Hilgert, Pasquale and Przebinda \cite{HiPaPz3,HiPaPz2, HiPaPz}. For the Laplacian acting on line bundles over complex hyperbolic spaces, the resonances has been completely determined by Will in \cite{Will}. The complete list of the resonances for the Laplacian acting on sections on $E_\tau$, when $G/K$ if of rank one and $\tau$ is arbitrary was determined in \cite{ROBY2021}.

With each resonance, one can associate a representation, called a \emph{residue representation} (see section \ref{section residue repr}). These representations have been determined for the scalar rank-one case by different method in \cite{MiaWill} and in \cite{HilgPasq}. Moreover, in \cite{Will}, this has been also done for the case of $\SL(2,\R)$, in the (line) bundle case. For the general rank one case, the residue representations were determined in \cite{ROBY2021} under the assumption that the resonances arise from the poles the trivial Plancherel density, which implies that $\tau$ has to occur in the spherical principal series representations. 

In this paper, we take on the case of $G/K$ of rank one and $\tau$ an arbitrary representation of $K$. 
Our main result is an algorithm, presented in section 5, which provides a simple way to compute the Langlands parameters, the Gelfand-Kirillov dimension and the wave front set, starting from the highest weight of $\tau$. 
Our methods to find these representations are based on the description of the composition series of the principal series representations, as one can find in the book of Collingwood \cite{CollinRROLG}. The results are recalled in section \ref{section about Col}. As an application of our algorithm, in section \ref{section pforms}, we compute the residue representations attached to the resonances of the Laplacian acting on the $p$-forms over $G/K$, for the classical rank one $G$ cases (all cases except the exceptional one of $\F_4$). Our choice of restricting ourself to the classical $G$'s is motivated by the significance of the hyperbolic spaces and not by any obstacle one could face in applying our algorithm. The quaternionic case is especially interesting because it presents situations in which $\tau$ has multiplicity 2 inside the principal series representations, the real and the complex cases being always multiplicity free. The exceptional case does not seem to give any other problems and the algorithm works completely.

\section{Notations and background}
\label{notations}

We shall use the standard notations $\Z_+,~ \Z,~\R, ~\C, ~\Hh, ~\Oo$ and $\C^\times$ for the nonnegative integers, the integers, the real numbers, the complex numbers, the quaternions, the octonions and the nonzero complex numbers. For a complex number $z \in \C$, we denote by $\Re(z)$ and $\Im(z)$ its real and imaginary parts, respectively. The normalization constants in the Haar measures do not matter in our computations. Hence, integrals have to be considered up to positive multiples. 

\subsection{Context} Let $G$ be a connected non-compact real semisimple Lie group with finite center and let $B(\cdot,\cdot)$ be the Killing form of the Lie algebra $\g$ of $G$. We denote by $\theta$ a Cartan involution of $\g$.
We denote by $\Ak$ the set of fixed points of $\theta$ and by $\p$ the eigenspace of $\theta$ for the eigenvalue $-1$. In other words: 
$$\Ak = \{X\in\g ~|~ \theta X = X\} ~~~\text{ and }~~~ \p = \{X\in\g ~|~ \theta X = -X\}~.$$
Then $\Ak$ is a Lie subalgebra of $\g$. The corresponding connected Lie subgroup of $G$ is maximal compact. We indicate it by $K$. 
The Cartan decomposition of the Lie algebra $\g$ is given by: $\g = \Ak \oplus \p$. 

Let $\Aa$ be a maximal abelian subspace of $\p$ and $A = \exp \Aa$ its associated Lie subgroup of $G$. The exponential map $\exp:\g \to G$ restricts to a diffeomorphism between $\Aa$ and $A$. The inverse map is the logarithm “$\log$”. In this paper, we are restricting ourself to real rank one groups $G$. In other words, we suppose that $\Aa$ is one-dimensional. 

Rank one symmetric spaces of the non-compact type are classified into three infinite families -- namely, the real, complex and quaternionic hyperbolic spaces -- and one exceptional example, the octonionic hyperbolic plane. In the following we will refer to these different cases respectively as the “real”, the “complex”, the “quaternionic”, and the “octonionic” case. The table at the end of subsection \ref{subsection Root system}  lists the groups $G$ and $K$, we choose in this paper, to realise each case. 

For every Lie algebra $\h$, we denote by $\h_\C$ its complexification, by $U(\h_\C)$ the universal enveloping algebra of $\h_\C$ and by $\Zentrum(\h_\C)$ the center of $U(\h_\C)$. 

\subsection{Root and restricted root systems} \label{subsection Root system} Let $\Aa^\ast$ be the vector space of linear forms on $\Aa$ and $\Aa_\C^\ast$ its complexification. The set $\Sigma$ of restricted roots of the pair $(\g,\Aa)$ consists of all nonzero linear forms $\alpha \in \Aa^\ast$ for which the vector space $$\g_\alpha:= \{X \in \g ~|~ [H,X] = \alpha(H)X \text{ , for every }H \in \Aa\}$$ contains nonzero elements. The dimension of $\g_\alpha$ is called the multiplicity of the root $\alpha$ and is denoted by $m_\alpha$. 

Let $\Sigma_+$ be a fixed set of positive restricted roots and let $\rho_\Aa:= \displaystyle\frac{1}{2}\sum_{\alpha \in \Sigma^+} m_\alpha \alpha $ be the half sum of the positive roots counted with their multiplicities.
Set $\n = \displaystyle\bigoplus_{\alpha \in \Sigma_+}\g_\alpha$ and $N$ the connected Lie subgroup of $G$ having $\n$ for Lie algebra.
According to the Iwasawa decomposition $G = KAN$, every element $x$ in $G$ can be uniquely written as \begin{equation}\label{eq Iwasawa}
   x = \IWk(x) e^{\IWH(x)} \IWn(x) 
\end{equation}
where $\IWk(x) \in K$, $\IWH(x) \in \Aa$ and $\IWn(x) \in N$. In the following, we set
\begin{equation}
    a^\lambda := e^{\lambda(\log a)}~ \text{ for } a\in A \text{ and }\lambda \in \Aa^*_\C~.
\end{equation}

Since $G$ is of real rank one, the set $\Sigma$ is either equal to $\{\pm \alpha \}$ or $\{\pm \alpha,\pm \alpha/2 \}$. Among the groups listed in Table \ref{Table cases for G and K}, only $G = \Spin(n,1)$ has restricted root system $\{\pm \alpha \}$. As a system of positive roots $\Sigma_+$ we choose $\{\alpha \}$ or $\{\alpha,\alpha/2\}$.
Then \mbox{$\rho_\Aa= \frac{1}{2}\big(m_\alpha + \frac{m_{\alpha/2}}{2} \big) \alpha$}, where we set $m_{\alpha/2} = 0$, if $\Sigma = \{\pm \alpha \}$.

The Killing form $B$ is positive definite on $\p$, so $\<X,Y\> :=B(X,Y)$ defines a Euclidean structure on $\p$ and on $\Aa \subset\p$. For all $\lambda\in\Aa^*$, let $H_\lambda$ denote the unique element in $\Aa$ such that $\<H_\lambda,H\> =\lambda(H)$ for all $H \in \Aa$. We extend the inner product to $\Aa^\ast$ by setting $\<\lambda,\mu\> :=\<H_\lambda,H_\mu\>$ for all $\lambda, ~\mu\in\Aa^*$. Further, we denote the $\C$-bilinear extension of $\<\cdot,\cdot\>$ on $\Aa^*$ to $\Aa_\C^*$ by the same symbol. 
We identify $\Aa_\C^*$ with $\C$ by means of the isomorphism: 

\begin{equation}
    \label{eq id Aa and C}\begin{array}{ccc}
    \Aa_\C^\ast & \longrightarrow & \C \\
    \lambda & \longmapsto & \lambda_\alpha:= \frac{\<\lambda,\alpha\>}{\<\alpha,\alpha\>}
\end{array}
\end{equation}
which identifies $\rho_\Aa$ with $\rho_\alpha:=\frac{1}{2}\left(m_\alpha+\frac{m_{\alpha/2}}{2}\right)$.

\vspace{5mm}
\begin{center}
 \begin{tabular}{|Sc|Sc|Sc|Sc|Sc|Sc|Sc|Sc|} 
 \hline
 Case &$G/K$ & $G$ & $K$ & $\Sigma^+$ & $m_{\alpha/2}$ & $m_\alpha$ & $\rho_\alpha$ \\
 \hline
 1 &$H^n(\R)$&$\Spin(2n,1)$ &$\Spin(2n)$ & $\{\alpha\}$ & $0$ & $2n-1$ & $n-\frac{1}{2}$ \\  
 \hline
2&$H^n(\C)$&$\SU(n,1)$ &$\Ss(\U(n) \times \U(1))$ & $\{\alpha/2,\alpha\}$ & $2n-2$ & $1$ & $\frac{n}{2}$\\ 
 \hline
3&$H^n(\Hh)$&$\Sp(n,1)$ &$\Sp(n)$ & $\{\alpha/2,\alpha\}$ & $4n-4$ & $3$ & $n+\frac{1}{2}$\\
 \hline
 4&$H^2(\Oo)$&$\F_{4~(-20)}$ &$\Spin(9)$ & $\{\alpha/2,\alpha\}$ & $8$ & $7$ & $\frac{11}{2}$\\
 \hline
\end{tabular}
\captionof{table}{Rank one Lie groups}\label{Table cases for G and K}
\end{center}

Let $M$ be the centralizer of $\Aa$ in $K$, $\m$ its Lie algebra, and let $\At$ be a Cartan subalgebra of $\m$. Then the Lie algebra $\h = \At\oplus\Aa$ is a Cartan subalgebra of $\g$.
The set $\Pi$ of roots of the pair $(\g_\C,\h_\C)$ consists of all nonzero linear forms $\varepsilon \in \h_\C^\ast$ for which the vector space
$$\g_\varepsilon:= \{X \in \g_\C ~|~ [H,X] = \varepsilon(H)X \text{ for every }H \in \h_\C\}$$
contains nonzero elements.
\newline We choose a set $\Pi^+$ of positive roots in $\Pi$ which is compatible with $\Sigma_+$, i.e. such that a root $\varepsilon\in \Pi$ is positive when $\varepsilon|_\Aa \in \Sigma_+$. We denote then by $\wt \Pi_+$ the corresponding positive Weyl chamber. 
Let also $\Pi_\Ak$ (respectively $\Pi_{\Ak}^{+}$) be the set of (positive) roots of the pair $(\Ak_\C, \h_\C| _{\Ak_\C})$ and $\Pi_\m$ (respectively $\Pi_{\m}^{+}$) the set of (positive) roots of the pair $(\m_\C, \At_\C)$. 
Finally, we denote the respective half sums of positive roots by $\rho$, $\rho_\Ak$ and $\rho_\m$. Recall the basic but important facts that $\rho = \rho_\Aa +\rho_\m$ and $\langle \rho_\Aa,\rho_\m\rangle = 0$. 

We denote by $\{\varepsilon_i\}_{i=1,\ldots,n}$ the usual dual basis of the Cartan Lie algebra of the pair $(\g_\C,\m_\C\oplus\Aa_\C)$. We recall the root system in each case:

\vspace{5mm}
\begin{center}
 \begin{tabular}{|Sc|Sc|Sc|Sc|} 
 \hline
 Case & $\Pi^+$ & $\Pi_{\Ak_\C}^+$ & $\Pi_{\m_\C}^+$\\
 \hline
 1 &$\begin{array}{cc}
       \varepsilon_i \pm \varepsilon_{j}, & ~1\leq i<j\leq n, \\
      \varepsilon_i, &~1\leq i\leq n
 \end{array}$&$\varepsilon_i \pm \varepsilon_{j}, ~1\leq i<j\leq n$&$\begin{array}{cc}
       \varepsilon_i \pm \varepsilon_{j}, & ~2\leq i<j\leq n, \\
      \varepsilon_i, &~2\leq i\leq n
 \end{array}$\\  
 \hline
2&$\varepsilon_i - \varepsilon_j, ~1\leq i<j\leq n+1$&$\varepsilon_i - \varepsilon_j, ~1\leq i<j\leq n$&$\varepsilon_i - \varepsilon_j, ~2\leq i<j\leq n$\\ 
 \hline
3&$\begin{array}{cc}
    \varepsilon_i \pm \varepsilon_{j}, &~1\leq i<j\leq n+1, \\2\varepsilon_i, &~1\leq i\leq n+1  
\end{array}$&$\begin{array}{cc}
    \varepsilon_i \pm \varepsilon_{j}, &~2\leq i<j\leq n+1, \\2\varepsilon_i, &~2\leq i\leq n+1 
\end{array}$ &$\begin{array}{cc}
    \varepsilon_i \pm \varepsilon_{j}, &~3\leq i<j\leq n+1, \\2\varepsilon_i, &~3\leq i\leq n+1 
\end{array}$\\
 \hline
 4&$\begin{array}{cc}
     \varepsilon_i \pm \varepsilon_j,&~ 1\leq i < j \leq 4\\\varepsilon_i, &~ 1\leq i \leq 4\\ \multicolumn{2}{c}{\frac{1}{2}( \varepsilon_1 \pm \varepsilon_2\pm \varepsilon_3\pm \varepsilon_4)}\end{array}$
     &$\begin{array}{cc}
     \varepsilon_i \pm \varepsilon_j,&~ 1\leq i < j \leq 4\\\varepsilon_i, &~ 1\leq i \leq 4 \end{array}$&$\begin{array}{cc}
     \varepsilon_i \pm \varepsilon_j,&~ 2\leq i < j \leq 4\\\varepsilon_i, &~ 2\leq i \leq 4 \end{array}$\\
 \hline
\end{tabular}
\captionof{table}{Positive root systems of rank one Lie groups}\label{Table cases for G and K}
\end{center}

\subsection{Homogeneous vector bundles}
Let $\hat{K}$ be the set of (equivalence classes of ) irreducible unitary representation of $K$ and let us fix $(\tau, V_\tau)\in \hat{K}$. 
Let $E_\tau:= G \times V_\tau/\sim$ denote the homogeneous vector bundle over $G/K$ associated with $\tau$. For the definition and properties of $E_\tau$, we refer the reader to \cite[\S 5.2 p. 114]{Wall1}. 
We write $\Gamma^\infty(E_\tau)$ for the space of all smooth sections of $E_\tau$. There is an isomorphism between $\Gamma^\infty (E_\tau)$ and the set of $\tau$-radial functions
\begin{equation*}
    C^\infty(G, \tau): = \{f:G \rightarrow V_\tau \text{ smooth}~| ~f(xk) = \tau(k^{-1}) f(x)~\text{ for all } x \in G \text{ and } k \in K \}
\end{equation*}

\subsection{Principal series representations} Let $\hat{M}$ be the set of all equivalence classes of irreducible unitary representations of $M$.
For $(\sigma,V_\sigma) \in \hat{M}$ and $\lambda \in \Aa^\ast_\C$, we denote by  $$\pi^\sigma_\lambda=\Ind{}^G_{MAN}(\sigma \otimes e^{i\lambda} \otimes \triv)$$ 
the representation of $G$ induced from $MAN$ by the representation $\sigma \otimes e^{i\lambda} \otimes \triv$ of $MAN$.
We will use the same notation for its derived representation of $\g$ too. The representation space $\Hi^\sigma_\lambda$ 
of $\pi^\sigma_\lambda$ is the Hilbert space completion of 
\begin{equation}\label{defprincseries}
    \{f: G \rightarrow V_\sigma ~|~ f(xman) = a^{-i\lambda-\rho}\sigma(m^{-1})f(x) ~\text{ for all }x \in G,~ m\in M,~ a\in A \text{ and }n\in N  \}
\end{equation}
with respect of the $L^2$ inner product
$$\<f,g\>_{\sigma} = \int_K\<f(k),g(k)\>_{V_\sigma}~dk\,,$$
where $\<\cdot,\cdot\>_{V_\sigma}$ is the inner product on $V_\sigma$ making $\sigma$ unitary.
The action of $\pi^\sigma_\lambda$ on $\Hi^\sigma_\lambda$ is given by 
$$\pi^\sigma_\lambda(g)f(x) := f(g^{-1}x)$$
for all $g,x \in G$ and $f\in \Hi^\sigma_\lambda$. The set $\{\pi^\sigma_\lambda ~|~\lambda\in \Aa_\C^\ast, \sigma \in \hat{M}\}$ is called the (minimal) principal series of $G$. 

The compact picture of the principal series representations is obtained by restriction of the elements of $\Hi^\sigma_\lambda$ to $K$. Its representation space, which we denote by $\Hi^\sigma$, is the Hilbert completion of: 
\begin{equation*}
    \{f: K \rightarrow V_\sigma ~|~ f(km) = \sigma(m^{-1})f(k) ~\text{ for all }k \in K,~ m\in M \}
\end{equation*}
with respect to $L^2$ inner product. It is independent of $\lambda$. The action is given by:
\begin{equation*}
    \pi^\sigma_\lambda(g)f(k) := e^{-(i\lambda+\rho)\IWH(g^{-1}k)} f(\IWk(g^{-1}k))
\end{equation*} 
for all $g \in G$, $k \in K$ and $f\in \Hi^\sigma$.
The representation $\pi^\sigma_\lambda$ is unitary for $\lambda \in i\Aa^*$. In the following, when working with principal series, we actually work with their Harish-Chandra modules.
The restriction of $\pi^\sigma_\lambda$ to $K$ is the representation $\Ind_M^K\sigma$ of $K$ induced from $\sigma$. In particular, because of Frobenius reciprocity theorem, for any $\tau \in \hat{K}$: \begin{equation*}
    m(\pi^\sigma_\lambda| _K, \tau,)=m(\tau| _M, \sigma)\,.
\end{equation*}
Here the symbol $m(\alpha,\beta)$ denotes the multiplicity of the irreducible representation $\beta$ in the representation $\alpha$.

We say that $\tau$ is a $K$-type of $\pi^\sigma_\lambda$ if it occurs in $\pi^\sigma_\lambda| _K$. We say that $\tau$ is a minimal $K$-type of an admissible representation $\pi$ of $G$ if and only if its highest weight $\mu$ minimizes the Vogan norm
\begin{equation}\label{eq Vogan norm}
    \lVert \mu\rVert_V = \langle \mu +2\rho_\Ak, \mu +2\rho_\Ak \rangle
\end{equation}
in the set of $K$-types of $\pi$. \cite[Theorem 1]{VoganProcAcad} ensures that each minimal $K$-type $\tau_{\min{}}$ has multiplicity  one in $\pi^\sigma_\lambda$. Therefore there exists a unique irreducible subquotient $J(\sigma,\lambda,\mu)$ of $\pi^\sigma_\lambda$ containing $\tau_{\min{}}$. 

Let $P_\tau$ denote the projection of $\Hi^\sigma_\lambda$ onto its subspace of vectors which transform under $K$ according to $\tau$,
that is,
    \begin{equation}
    \label{Ptau}
P_\tau=d_\tau \int_K \pi^\sigma_{\lambda}(k)\gamma_\tau(k^{-1})\, dk\,.
    \end{equation}
The spherical function $\varphi_\tau^{\sigma,\lambda}$ is defined as the $\End(V_\tau)$-valued function on $G$ given by
    \begin{equation}
\varphi_\tau^{\sigma,\lambda}(x):=\varphi_\tau^{\pi^\sigma_{\lambda}}(x):=d_\tau \int_K \tau(k) \psi_\tau^{\sigma,\lambda}(xk^{-1})\, dk\,,  
    \end{equation}
where 
    \begin{equation}
    \label{psi-tausigma}
\psi_\tau^{\sigma,\lambda}(x)=\Tr\big(P_\tau \pi^\sigma_{\lambda}(x) P_\tau\big)\,.
    \end{equation}

\subsection{Homogeneous differential operators}\label{subsection hom diff op} A homogeneous differential operator $D$ on $E_\tau$ is a linear differential operator from $\Gamma^\infty(E_\tau)$ to itself which is invariant under the $G$-action $L$ by left translations, that is
\begin{equation}
    L(g)D = DL(g)\quad\text{ for all }g \in G~.
\end{equation}
The set of homogeneous differential operators on $E_\tau$ is an algebra with respect to composition. We denote it by $\D(E_\tau)$. 
It acts on $C^\infty(G, \tau)$ because of the isomorphism with the space smooth sections $\Gamma^\infty(E_\tau)$. Unlike in the scalar case, i.e. when $\tau$ is the trivial representation, this algebra need not be commutative. Conditions equivalent to the commutativity of $\D(E_\tau)$ are stated in \cite[Proposition 2.2]{Camp4} and \cite[Proposition 3.1]{Ricci}. In the rank one case, this algebra is always commutative when $G$ is $\Spin(n,1)$ or $\SU(n,1)$. See for instance \cite[Theorem 2.3]{Camp4}. The structure of $\D(E_\tau)$ can be found in \cite[Section 2.2]{Olb}.

Let $U(\g_\C)$ be the universal enveloping algebra of the complexification $\g_\C$ of $\g$. Each element of $U(\g_\C)$ induces a left-invariant differential operator on $G$ by:
\begin{equation}
    \big(X_1 \cdots X_k\cdot f\big)(g): = \frac{\partial}{\partial t_1}\frac{\partial}{\partial t_2} \cdots \frac{\partial}{\partial t_k} f (g \exp t_1X_1 \exp t_2X_2 \cdots \exp t_kX_k ) \Big| _{t_1=\ldots= t_k =0}
\end{equation}
for all $X = X_1 \cdots X_k \in U(\g_\C)$, $f \in C^\infty(G)$ and $g\in G$.\newline
Let $U(\g_\C)^K$ denote the subalgebra of the elements in $U(\g_\C)$ which are invariant under the adjoint action $\Ad$ of $K$. The elements of $U(\g_\C)^K$ act on on $C^\infty(G,\tau)$ as homogeneous differential operators. As $K$ is compact, Theorem 1.3 in \cite{Minemura} ensures that all elements of $\D(E_\tau)$ can be written as an element of $U(\g_\C)^K$. But there is no isomorphism in general. \newline
We can extend the action of $U(\g_\C)^K$ to the set of radial systems of section $C^\infty(G,K,\tau, \tau)$ by setting:
\begin{equation*}
    \big(D\cdot \phi \big) v := D\cdot (\phi \cdot v)
\end{equation*}
for all $D \in U(\g_\C)^K$, $\phi \in C^\infty(G,K, \tau, \tau)$ and $v\in V_\tau$.\newline
\subsection{The Laplace operator}\label{subsection Laplace} Let $\{X_1,\ldots, X_{\dim \g}\}$ be any basis of $\g$. We denote by $g^{ij}$ the $ij$-th coefficient of the inverse of the matrix $\big( B(X_i,X_j)\big)_{1\leq i,j\leq \dim \g}$, where $B$ is the Killing form. The Casimir operator is defined by
\begin{equation*}
\Omega := \sum_{1\leq i,j\leq \dim \g} g^{ij}X_jX_i ~.
\end{equation*}
If $\big(X_k\big)_{k=1,\ldots,\dim \Ak}$ and $\big(X_k\big)_{k=\dim \Ak + 1,\ldots,\dim \g}$ are respectively orthonormal basis of $\Ak$ and $\p$ with respect to $B_\theta$, then: 
\begin{equation*}
    \Omega = -\sum_{i=1}^{\dim \Ak} X_i^2 + \sum_{i=\dim \Ak+ 1}^{\dim \g} X_i^2~.
\end{equation*}
In fact, $\Omega$ is in the center of $U(\g_\C)$. The invariant differential operator corresponding $-\Omega$ is the positive Laplacian $\Delta$.

We can extend any representation of $\g$ to $\g_\C$ by linearity and to a representation of the associative algebra $U(\g_\C)$. These representations will always be denoted by the same symbol. 
Since $\Omega$ is in the center of $U(\g_\C)$, the linear operator $\pi^\sigma_\lambda(\Omega)$ is an interwining operator of the representation $\pi^\sigma_\lambda$ for all $\lambda \in \Aa^\ast_\C$ and $\sigma \in \hat{M}$. Lemma 4.1.8 in \cite{VoganRRRLG} ensures that $\pi^\sigma_\lambda(\Omega)$ acts by a scalar. 
To compute this scalar, one can use \cite[Proposition 8.22 and Lemma 12.28]{Kna1}, and get that:
    \begin{equation}\label{Multope}
    \pi^\sigma_\lambda(\Omega) = \Big(-\langle\lambda,\lambda\rangle - \langle\rho,\rho\rangle + \langle\mu_\sigma +\rho_\m,\mu_\sigma+\rho_\m\rangle\Big) \Id =: -M(\sigma,\lambda)\Id~.
    \end{equation}
Here $\mu_\sigma$ is the highest weight of $\sigma$.

\section{The residue representations}

\label{section residue repr}
In this section, we recall the results of \cite[Section 4]{ROBY2021}. With each resonance of the Laplace operator on homogeneous vector bundle over rank one symmetric space, is associated a representation, called a residue representation, which we are going to describe.

We suppose that the highest weight of $\tau\in \hat{K}$ is known. Then using \cite{Bald1}, one can find the $M$-types of $\tau$. We denote by $\hat{M}(\tau)$ their set and by $\#\hat{M}(\tau)$ the cardinality of $\hat{M}(\tau)$. Let us indicate the elements of $\hat{M}(\tau)$ by $\sigma_j$, with $j=1,\ldots,\#\hat{M}(\tau)$, and let $\mu_{\sigma_j}$ be the highest weight of $\sigma_j$.
One can use now \cite{Mia} to find the poles of the Plancherel density $p_{\sigma_j}$ associated with $\sigma_j$. These poles can be either in the sets $i\Z$ or $i(\Z+\frac12)$, can be indexed by $\Z$ and their set is symmetric with respect to 0. To each pole  $\pm \lambda_k^{\sigma_j}$, $k\in\Z_+$, of $p_{\sigma_j}$ (we choose $\lambda_k^{\sigma_j}$ to be negative without loss of generality), i.e. to each resonance arising from the restriction of $\Delta$ to $\bigcup\limits_{\lambda\in\Aa^*} \Hi_\lambda^{\sigma_j}$, we attach a so-called \emph{residue representation}. 
Before describing these representations, let us recall first the theorem which gives the resonances of the Laplace operator $\Delta$ on the homogeneous vector bundle $E_\tau$. The resonances are the poles of the meromorphic continuation of the resolvent of $\Delta$, defined for every $z\in \C\setminus\R$ by
\begin{equation}
    R(z) = (\Delta-z)^{-1}~,
\end{equation}
once $R(z)$ is considered as a non-selfadjoint operator on the space $C_c^\infty(G,\tau)$ of smooth compactly supported sections of $E_\tau$. 
we decompose $R$ as follows:
\begin{equation}
\label{Rtau}
    R=\sum_{\sigma\in \hat{M}(\tau)} d_\sigma R_\sigma\,,
\end{equation}
where 
\begin{equation*}R_\sigma(z) := \int_{\Aa^*}  (M(\sigma,\lambda) - z)^{-1} \Big(\varphi_\tau^{\sigma,\lambda} \ast f  \Big)(x) ~ p_\sigma(\lambda) ~ d\lambda \end{equation*}

\begin{Thm}[Theorem 1 in \cite{ROBY2021}]
\label{Thm mer ext}
Let
\begin{equation}\label{zetasigma}
    \zeta_\sigma := \sqrt{z-\langle\rho,\rho\rangle + \langle\mu_\sigma +\rho_M,\mu_\sigma+\rho_M\rangle)}
\end{equation}
Here $\sqrt{\cdot}$ denotes the single-valued branch of the square root function determined on $\C\setminus[0,+\infty[$ by the condition $\sqrt{-1} = -i$.

Let 
$$S=\left\{(z,\zeta) \in \C^2~|~\zeta^2 := z-\langle\rho,\rho\rangle + \langle\mu_\sigma +\rho_M,\mu_\sigma+\rho_M\rangle\right\}~.$$
Then the restriction of the resolvent of the $\Delta$ to $C_c^\infty(G,\tau)$ extends meromorphically from $S^+ = \{(z,\zeta) \in S~|~\Im(\zeta)>0\}$ to $S$, which holds by the following formula, up to constants and for every $N\in \Z_+$,

\begin{multline}\label{eq meromorphiccontinuation}
\left(R_{\sigma_j}(\zeta_{\sigma_j})f\right)(x) =
\int_{\R-i(N+1/4)}\frac{1}{\lambda|\alpha|-\zeta_{\sigma_j}} \Big(\varphi_\tau^{\sigma_j,\lambda\alpha} \ast f  \Big)(x) ~ \frac{p_{\sigma_j}(\lambda\alpha)}{\lambda} ~ d\lambda \\
+~~~
\sum_{\substack{k\in \Z_+ \\ 0>\Im(\lambda^{\sigma_j}_k) > -N+1/4}} \frac{1}{\lambda^{\sigma_j}_k|\alpha|-\zeta_{\sigma_j}} \Big(\varphi_\tau^{\sigma_j,\lambda^{\sigma_j}_k\alpha} \ast f \Big)(x) ~ \Res_{\lambda=\lambda^{\sigma_j}_k}\frac{p_{\sigma_j}(\lambda\alpha)}{\lambda}~.
\end{multline}

The resonances of $\Delta$ acting on $C_c^\infty(G,\tau)$ are the (simple) poles of this extension and are given by the pairs
\begin{equation}
    (z_{\sigma_j,k}, \zeta_{\sigma_j,k}) = \big({(\lambda_k^{\sigma_j}|\alpha|)}^2 +\langle\rho,\rho\rangle - \langle\mu_{\sigma_j} +\rho_\m,\mu_{\sigma_j}+\rho_\m\rangle~, \lambda_k^{\sigma_j}|\alpha|~\big)
\end{equation}
where $k\in \Z_+$.
\end{Thm}

\begin{Rem}
In Theorem \ref{Thm mer ext}, we (choose to) extend the holomorphic part of the resolvent on $S^+$. This is why only the negative poles of the Plancherel density can become poles of the meromorphic extension of $R$. Of course one can choose to extend the holomorphic part of the resolvent on
$S^- = \{(z,\zeta) \in S~|~\Im(\zeta)<0\}.$
This is where the parity of the Plancherel density is really important. Its positive poles are just the opposite of the negative poles. The two extensions of $R$ appear thus to be equivalent. 
\end{Rem}

One can see that for each negative pole of the Plancherel densities, we get a representation, which come from the left action of $G$ on the residue in \eqref{eq meromorphiccontinuation}. The representation space is
\begin{equation} \label{eq residue representation}
    \E_k^{\sigma_i} := \{\varphi_{\tau}^{\sigma_i,\lambda_k^{\sigma_i}} \ast f ~|~ f \in  C_c^\infty(G,\tau)\}~.
\end{equation}
The left action of $G$ on $\E_k^{\sigma_i}$ is called the \emph{residue representation} arising from the pole $\lambda_k^{\sigma_i}$ of the Plancherel density $p_{\sigma_i}$. 
This is exactly the image of the \emph{residue operator} $\RR_k^{\sigma_i}$ which is defined as follows. 
\begin{equation}
    \fonction{\RR^{\sigma_i}_k}{C_c^\infty(G,\tau)}{C^\infty(G,\tau)}{f}{\varphi_{\tau}^{\sigma_i,\lambda^i_k\alpha} \ast f}~.
\end{equation} 

To identify the residue representation among the representations of $G$, the idea is to embed $\E_k^{\sigma_i}$ in a principal series representation as follows. Let $T_j^{i,k}$ be the map defined, for each $j=1,...,m(\sigma_i,\tau)$, by
\begin{equation}
    \fonction{T_j^{i,k}}{C_c^\infty(G,\tau)}{\Hi^{\sigma_i}_k}{f}{\left[x\mapsto \int_G \pi^{\sigma_i}_k(g)\big(P_j^\ast f(g)\big)(x) ~dg\right]~.}
\end{equation}
Here $P_j$ is the projection onto the $j$-th $\tau$-isotypic component in $\Hi^{\sigma_i}_k$ and $\ast$ denotes the Hermitian adjoint. There are $m(\sigma_i,\tau)$ maps $T_j^{i,k}$, for $i$ and $k$ fixed.
The image of $T_j^{i,k}$ is the closure of the space spanned by the left translates of $P_j^*V_{\tau}$ (see \cite[Lemma 4.1]{ROBY2021}). 
This map allows us to decompose the residue operators $\RR^{\sigma_i}_k$ as the composition of $m(\sigma_i,\tau) +1$ operators:
\begin{equation}
    \begin{array}{cll}
        \RR^{\sigma_i}_k : C_c^\infty(G,\tau) & \rightarrow ~~~~~~~~\Hi^{\sigma_i}_k &\rightarrow  \E_k^{\sigma_i}\\
         ~~f                  & \mapsto ~~~~~~~~T_j^{i,k}(f)                     &\mapsto  \sum_j P_j ~\pi^{\sigma_i}_k((\cdot)^{-1})\big(T_j^{i,k}(f)\big)
    \end{array}~.
\end{equation}

In this way, we can use the structure of the principal series representations $\Hi^{\sigma_i}_k$ to study the residue representations $\E_k^{\sigma_i}$. 

\section{The structure of the principal series representations of a real rank-one classical Lie group}
\label{section about Col}
In this section, we collect some results and notations we need to understand the principal series representations of a real rank one classical Lie group. Our main reference is the book \cite{CollinRROLG}.

\subsection{Some notations}

Each infinitesimal character $\chi$ will be indicated by $\chi_\gamma$ where $\gamma\in\h_\C^*$. In fact, if we denote by $HC$ the Harish-Chandra isomorphism, then every character $\chi$ of $\Zentrum(\g)$ has the form
$$\chi(Z) = \chi_\gamma(Z) = HC(Z)(\gamma)~$$
for some $\gamma \in \h_\C^*$. Moreover $\chi_\gamma =\chi_{\gamma'}$ if and only if $\gamma'$ and $\gamma$ are in the same $W_\C$-orbit in $\h_\C^*$. Here $W_\C$ denotes the Weyl group of $(\g_\C,\h_\C)$.

We denote by $\HC$ the category of Harish-Chandra modules. Moreover, we indicate by $\HC(\lambda)$ the subcategory of Harish-Chandra modules with generalized infinitesimal character $\chi_\lambda$. Here, by “generalized infinitesimal character $\chi_\lambda$”, we mean that if $V\in\HC(\lambda)$, then for every $Z\in \Z(\g)$, $Z-\chi_\gamma(Z)$ acts nilpotently. For every $\gamma\in \h^*$ regular, we define the projection functor $p_\gamma$ from $\HC$ to $\HC(\lambda)$ as follows. For every $V\in\HC $, $p_\gamma V$ is the maximal subspace in $V$ such that for every $Z\in \Z(\g)$, $Z-\chi_\gamma(Z)$ acts nilpotently. 

Let $\Lambda$ be the lattice of weights of the finite dimensional representations of $G$. For each $\mu\in \Lambda$, we denote by $F^\mu$ the irreducible finite dimensional $G$-module of highest weight $\mu$ and set $F_{-\mu} := F^*_\mu$ for the irreducible finite dimensional one of lowest weight $-\mu$. Then for every $\lambda$ in the positive Weyl chamber, we can define the two functors $\Phi_{\lambda+\mu}^\lambda$ and $\Psi^{\lambda+\mu}_\lambda$ by
\begin{align}\label{eq Def up/down functors}
    \Phi_{\lambda+\mu}^\lambda &:= p_{\lambda+\mu} \circ [(...)\otimes F^\mu]\circ p_\lambda\\
    \Psi^{\lambda+\mu}_\lambda &:= p_\lambda\circ [(...)\otimes F_{-\mu}]\circ p_{\lambda+\mu}
\end{align}

Let $\h$ be a split or a compact Cartan subalgebra of $\g$. By \emph{regular character} of $H = \exp\h$ (see \cite[paragraph 6, page 409]{VoganRRRLG} or \cite[page 48]{CollinRROLG}), we define a pair $(\Gamma,\gamma)$, verifying 
\begin{enumerate}
    \item $\Gamma$ is an irreducible character of $H$,
    \item $\gamma$ is a regular element of $\h^*_\C$.
\end{enumerate}
If $\h$ is a compact Cartan subalgebra, we require additionally that $d\Gamma = \gamma+\rho-2\rho_\Ak$. If $\h$ is the split Cartan subalgebra, we ask that $d\Gamma = \gamma-\rho_\m$. 

For each regular character $(\Gamma,\gamma)$ on the split (respectively compact) Cartan subalgebra, we denote by $\pi(\gamma)$ the equivalence class of the principal (respectively discrete) series representations with regular character $(\Gamma,\gamma)$. Each $\pi(\gamma)$ is a standard module. One can prove (see \cite[2.1.10 and 2.1.11]{CollinRROLG}) that each standard module $\pi(\gamma)$ admits a unique irreducible quotient module, which we will denote by $\overline{\pi}(\gamma)$.


For a fixed regular character $(\Gamma,\gamma)$, we indicate by $\Pi_\gamma^+$ the set of positive roots for which $\gamma$ is dominant and by $\rho_\gamma$ the half sum of the positive roots with respect to $\Pi_\gamma^+$.




\subsection{Some useful theorems}
\label{section Col's Thm used}

The first important result allows us to reduce the study of every principal series representation to the case of the trivial infinitesimal character. 

\begin{Thm}[Theorem 4.3.1 in \cite{CollinRROLG}]\label{Thm Col: red inf ch}
Let $\gamma\in \h_\C$ be strictly dominant and let $\mu$ be a highest weight with respect to $\Pi_\gamma$ such that $\gamma-\mu = \rho_\gamma$. 
Then $\Psi$ sets up a bijective correspondence between composition factors of $\pi(\gamma)$ and $\pi(\rho_\gamma)$. 
\end{Thm}

The next theorems describe the regular characters of trivial infinitesimal character for the three rank-one classical Lie groups as well as the decomposition of their corresponding principal series representations and their Gelfand-Kirillov dimension. 
The composition series is described by superposition of boxes following the following rules :
\begin{enumerate}
    \item Each box realizes a subquotient of the principal series representations,
    \item the box at the top is the (unique) maximal quotient of the principal series representation, 
    \item the $U(\g_\C)$-module generated by a box is the composition series of every box which is above it. 
\end{enumerate}
In the following theorems, let $\pi_{ij}:= \pi(\gamma_{ij})$. Recall that the $e_j$ are the usual dual basis of that of the compact Cartan subalgebra.

\begin{Thm}[for $G=\Spin(2n,1)$, see page 81, page 208 and Theorem 5.2.4]\label{Thm Col: Spin}
Trivial regular characters for the principal series representations, for $1\leq i \leq n$:
\begin{equation}
    2\gamma_{0i} := (2n-2i+1) \varepsilon_1 + \sum_{j=2}^i (2n-2j+3)\varepsilon_j + \sum_{j=i+1}^n (2n-2j+1)\varepsilon_j 
\end{equation}
Trivial regular characters for the discrete series representations: 
\begin{equation}
    2\gamma_0 := \sum_{j=1}^n (2n-2j+1)e_j ~~~~~\text{and}~~~~~ 2\gamma_1 := \sum_{j=1}^{n-1} (2n-2j+1)e_j -e_n
\end{equation}
The corresponding principal series decompositions, for $1\leq i \leq n-1$:
\begin{center}
\begin{tikzpicture}[scale = 0.7]
\draw (-1,0) node  {$\pi_{0,i} =$};
\draw [] (0,0) rectangle (2,1);
\draw [] (0,0) rectangle (2,-1);
\draw [] (0.2,0) to (0.2,-1);
\draw (1,0.5) node  {$\overline{\pi}_{0,i}$};
\draw (1.2,-0.5) node  {$\overline{\pi}_{0,i+1}$};

\begin{scope}[xshift=7 cm] 
\draw (-1,0) node  {$\pi_{0,n} =$};
\draw [] (0,0) rectangle (4,1);
\draw [] (0,0) rectangle (4,-1);
\draw [] (1.7,0) to (1.7,-1);
\draw [] (2.3,0) to (2.3,-1);
\draw (2,0.5) node  {$\overline{\pi}_{0,n}$};
\draw (0.8,-0.5) node  {$\overline{\pi}_{0}$};
\draw (3.2,-0.5) node  {$\overline{\pi}_{1}$};
\draw (2,-0.5) node  {$\oplus$};
\end{scope}
\end{tikzpicture}
\end{center}
where
$\overline{\pi}_{0,1}$ is finite dimensional and $\overline{\pi}_{0,i}$ for $2\leq i\leq n$, ${\pi}_{0}$ and ${\pi}_{1}$ have Gelfand-Kirillov dimension $2n-1$. 
\end{Thm}

\begin{Thm}[for $G=\SU(n,1)$, see page 83, page 208 and Theorem 5.3.1]\label{Thm Col: SU}
Trivial regular characters for the principal series representations, for $0\leq i \leq n-1$ and $1\leq j \leq n-i$:
\begin{equation}
    2\gamma_{ij} := (n-2i) \varepsilon_1 + \sum_{l=2}^{i+1} (n-2l+4) \varepsilon_l + \sum_{l=i+2}^{n-j+1} (n-2l+2) \varepsilon_l + \sum_{l=n-j+2}^{n} (n-2l) \varepsilon_l + (n-2(n-j+2)+2) \varepsilon_{n+1}
\end{equation}
Trivial regular characters for the discrete series representations, for $0\leq i \leq n$: 
\begin{equation}
    2\gamma_i := \sum_{l=1}^i (n-2l+2)e_l + \sum_{l=i+1}^n (n-2l)e_l + (n-2i) e_{n+1} 
\end{equation}
The corresponding principal series decompositions:\vspace{0.5cm}

\begin{minipage}{0.33\textwidth}
\begin{center}
{\footnotesize For $0\leq i\leq n-1$,}
\vspace{0.2cm}

\begin{tikzpicture}[scale = 0.7]
\draw (-1,0) node  {$\pi_{i,n-i} =$};
\draw [] (0,0) rectangle (4,1);
\draw [] (0,0) rectangle (4,-1);
\draw [] (1.7,0) to (1.7,-1);
\draw [] (2.3,0) to (2.3,-1);
\draw (2,0.5) node  {$\overline{\pi}_{i,n-i}$};
\draw (0.8,-0.5) node  {$\overline{\pi}_{i}$};
\draw (3.2,-0.5) node  {$\overline{\pi}_{i+1}$};
\draw (2,-0.5) node  {$\oplus$};
\end{tikzpicture}
\end{center}
\end{minipage}
\begin{minipage}{0.33\textwidth}
\begin{center}
{\footnotesize For $0\leq i\leq n-2$, \\$1\leq j\leq n-i-2$,}
\vspace{0.2cm}

\begin{tikzpicture}[scale = 0.7]
\draw (-1,-0.5) node  {$\pi_{i,j} =$};
\draw [] (0,0) rectangle (4,1);
\draw [] (0,0) rectangle (4,-1);
\draw [] (1.7,0) to (1.7,-1);
\draw [] (2.3,0) to (2.3,-1);
\draw [] (0,-1) rectangle (4,-2);
\draw (2,0.5) node  {$\overline{\pi}_{i,j}$};
\draw (0.8,-0.5) node  {$\overline{\pi}_{i+1,j}$};
\draw (3.2,-0.5) node  {$\overline{\pi}_{i,j+1}$};
\draw (2,-0.5) node  {$\oplus$};
\draw (2,-1.5) node  {$\overline{\pi}_{i+1,j+1}$};
\end{tikzpicture}
\end{center}
\end{minipage}
\begin{minipage}{0.33\textwidth}
\begin{center}
{\footnotesize For $0\leq i\leq n-2$,\\ $j = n-i-1$,} 
\vspace{0.2cm}

\begin{tikzpicture}[scale = 0.7]
\draw (-1,-0.5) node  {$\pi_{i,j} =$};
\draw [] (0,0) rectangle (4,1);
\draw [] (0,0) rectangle (4,-1);
\draw [] (1.7,0) to (1.7,-1);
\draw [] (2.3,0) to (2.3,-1);
\draw [] (0,-1) rectangle (4,-2);
\draw (2,0.5) node  {$\overline{\pi}_{i,j}$};
\draw (0.8,-0.5) node  {$\overline{\pi}_{i+1,j}$};
\draw (3.2,-0.5) node  {$\overline{\pi}_{i,j+1}$};
\draw (2,-0.5) node  {$\oplus$};
\draw (2,-1.5) node  {$\overline{\pi}_{i+1}$};
\end{tikzpicture}
\end{center}
\end{minipage}

\vspace{2mm}
where $\overline{\pi}_{0,1}$ is finite dimensional. 
\newline $\overline{\pi}_{0,j}$ for $2\leq j\leq n$, $\overline{\pi}_{i,1}$ for $1\leq i\leq n-1$, ${\pi}_{0}$ and ${\pi}_{n}$ have Gelfand-Kirillov dimension $n$. 
\newline $\overline{\pi}_{i,j}$ for  $1\leq i\leq n-1$ and $2\leq j\leq n$, and ${\pi}_{l}$ for $1\leq l\leq n-1$ have Gelfand-Kirillov dimension $2n-1$. 
\end{Thm}

\begin{Thm}[(for $G=\Sp(n,1)$, page 85 and Theorem 5.4.1)]\label{Thm Col: Sp}
Trivial regular characters for the principal series representations, for $0\leq i < j\leq n$:
\begin{equation}
    \gamma_{ij} := (n+1-i) \varepsilon_1 +(n+1-j) \varepsilon_2 + \sum_{l=3}^{i+2} (n+4-l) \varepsilon_l + \sum_{l=i+3}^{j+1} (n+3-l) \varepsilon_l + \sum_{l=j+2}^{n+1} (n+2-l) \varepsilon_l 
\end{equation}
and for $0\leq i \leq 2n-j\leq n-1$:
\begin{equation}
    \gamma_{ij} := (n+1-i) \varepsilon_1 + (n-j) \varepsilon_2 + \sum_{l=3}^{i+2} (n+4-l) \varepsilon_l + \sum_{l=i+3}^{2n-j+2} (n+3-l) \varepsilon_l + \sum_{l=2n-j+3}^{n+1} (n+2-l) \varepsilon_l~.
\end{equation}
Trivial regular characters for the discrete series representations: 
\begin{equation}
    2\gamma_i := \sum_{l=1}^i (n-l+2)e_l + \sum_{l=i+1}^n (n-l+1)e_l + (n-i+1) e_{n+1} 
\end{equation}
The corresponding principal series decompositions: \vspace{0.5cm}

\begin{minipage}{0.3\textwidth}

\begin{center}
{\footnotesize For $0\leq i\leq n-1$,}
\vspace{0.2cm}

\begin{tikzpicture}[scale = 0.7]
\draw (-1.2,0) node  {$\pi_{i,2n-i} =$};
\draw [] (0,0) rectangle (4,1);
\draw [] (0,0) rectangle (4,-1);
\draw [] (1.7,0) to (1.7,-1);
\draw [] (2.3,0) to (2.3,-1);
\draw (2,0.5) node  {$\overline{\pi}_{i,2n-i}$};
\draw (0.8,-0.5) node  {$\overline{\pi}_{i}$};
\draw (3.2,-0.5) node  {$\overline{\pi}_{i+1}$};
\draw (2,-0.5) node  {$\oplus$};
\end{tikzpicture}
\end{center}
\end{minipage}
$\left|
\begin{minipage}{0.3\textwidth}
\begin{center}
{\footnotesize For $0\leq i\leq n-2$,}
\vspace{0.2cm}

\begin{tikzpicture}[scale = 0.7]
\draw (-1.5,-0.5) node  {$\pi_{i,2n-i-1} =$};
\draw [] (0,0) rectangle (6,1);
\draw [] (0,0) rectangle (6,-1);
\draw [] (3.5,0) to (3.5,-1);
\draw [] (4.1,0) to (4.1,-1);
\draw [] (0,-1) rectangle (6,-2);
\draw (3,0.5) node  {$\overline{\pi}_{i,2n-i-1}$};
\draw (1.8,-0.5) node  {$\overline{\pi}_{i+1,2n-i-1}$};
\draw (5.1,-0.5) node  {$\overline{\pi}_{i,2n-i}$};
\draw (3.8,-0.5) node  {$\oplus$};
\draw (3,-1.5) node  {$\overline{\pi}_{i+1}$};
\end{tikzpicture}
\end{center}
\end{minipage}\hspace{15mm}\right|$
\begin{minipage}{0.33\textwidth}
\begin{center}
\textcolor{white}{kk}
\vspace{0.2cm}

\begin{tikzpicture}[scale = 0.7]
\draw (-1.2,0) node  {$\pi_{n-1,n} =$};
\draw [] (0,0) rectangle (3,1);
\draw [] (0,0) rectangle (3,-1);
\draw [] (0.2,0) to (0.2,-1);
\draw (1.5,0.5) node  {$\overline{\pi}_{n-1,n}$};
\draw (1.7,-0.5) node  {$\overline{\pi}_{n-1,n+1}$};
\end{tikzpicture}
\end{center}
\end{minipage}

\vspace{0.3cm}
\begin{minipage}{0.6\textwidth}
\begin{center}
\begin{tikzpicture}[scale = 0.7]
\draw (-1.5,-0.5) node  {$\pi_{n-2,n} =$};
\draw [] (0,0) rectangle (9,1);
\draw [] (0,0) rectangle (9,-1);
\draw [] (3.4,0) to (3.4,-1);
\draw [] (4,0) to (4,-1);
\draw [] (6.7,0) to (6.7,-1);
\draw [] (6.1,0) to (6.1,-1);
\draw [] (0,-1) rectangle (9,-2);
\draw (4.5,0.5) node  {$\overline{\pi}_{n-2,n}$};
\draw (1.8,-0.5) node  {$\overline{\pi}_{n-2,n+1}$};
\draw (5.1,-0.5) node  {$\overline{\pi}_{n-1,n}$};
\draw (7.8,-0.5) node  {$\overline{\pi}_{n}$};
\draw (3.7,-0.5) node  {$\oplus$};
\draw (6.4,-0.5) node  {$\oplus$};
\draw (4.5,-1.5) node  {$\overline{\pi}_{n-1,n+1}$};
\end{tikzpicture}
\end{center}
\end{minipage}
\begin{minipage}{0.33\textwidth}
\begin{center}
\begin{tikzpicture}[scale = 0.7]
\draw (-1.4,-0.5) node  {$\pi_{n-2,n-1} =$};
\draw [] (0,0) rectangle (3,1);
\draw [] (0,0) rectangle (3,-1);
\draw [] (0,-1) rectangle (3,-2);
\draw [] (0.2,0) to (0.2,-1);
\draw [] (0.4,-2) to (0.4,-1);
\draw (1.5,0.5) node  {$\overline{\pi}_{n-2,n-1}$};
\draw (1.7,-0.5) node  {$\overline{\pi}_{n-2,n}$};
\draw (1.9,-1.5) node  {$\overline{\pi}_{n}$};
\end{tikzpicture}
\end{center}
\end{minipage}

\vspace{0.3cm}
\begin{minipage}{0.5\textwidth}
\begin{center}
{\footnotesize For $0\leq i\leq n-3$,~~$i+1\leq j\leq 2n-i-2$\\ and $j-i\geq 3$,}
\vspace{0.2cm}

\begin{tikzpicture}[scale = 0.7]
\draw (-1,-0.5) node  {$\pi_{i,j} =$};
\draw [] (0,0) rectangle (4,1);
\draw [] (0,0) rectangle (4,-1);
\draw [] (1.7,0) to (1.7,-1);
\draw [] (2.3,0) to (2.3,-1);
\draw [] (0,-1) rectangle (4,-2);
\draw (2,0.5) node  {$\overline{\pi}_{i,j}$};
\draw (0.8,-0.5) node  {$\overline{\pi}_{i+1,j}$};
\draw (3.2,-0.5) node  {$\overline{\pi}_{i,j+1}$};
\draw (2,-0.5) node  {$\oplus$};
\draw (2,-1.5) node  {$\overline{\pi}_{i+1,j+1}$};
\end{tikzpicture}
\end{center}
\end{minipage}
\begin{minipage}{0.5\textwidth}
\begin{center}
{\footnotesize For $0\leq i\leq n-3$,~~$i+1\leq j\leq 2n-i-2$\\ and $j-i=1$,}
\vspace{0.2cm}

\begin{tikzpicture}[scale = 0.7]
\draw (-1.4,-0.5) node  {$\pi_{i,j} =$};
\draw [] (0,0) rectangle (3,1);
\draw [] (0,0) rectangle (3,-1);
\draw [] (0,-1) rectangle (3,-2);
\draw [] (0.2,0) to (0.2,-1);
\draw [] (0.4,-2) to (0.4,-1);
\draw (1.5,0.5) node  {$\overline{\pi}_{i,j}$};
\draw (1.7,-0.5) node  {$\overline{\pi}_{i,j+1}$};
\draw (1.9,-1.5) node  {$\overline{\pi}_{i+2;j+2}$};
\end{tikzpicture}
\end{center}
\end{minipage}

\begin{minipage}{0.6\textwidth}
\begin{center}
{\footnotesize For $0\leq i\leq n-3$,~~$i+1\leq j\leq 2n-i-2$ and $j-i=2$,}
\vspace{0.2cm}

\begin{tikzpicture}[scale = 0.7]
\draw (-1.5,-0.5) node  {$\pi_{i,j} =$};
\draw [] (0,0) rectangle (9,1);
\draw [] (0,0) rectangle (9,-1);
\draw [] (3.4,0) to (3.4,-1);
\draw [] (4,0) to (4,-1);
\draw [] (6.7,0) to (6.7,-1);
\draw [] (6.1,0) to (6.1,-1);
\draw [] (0,-1) rectangle (9,-2);
\draw (4.5,0.5) node  {$\overline{\pi}_{i,j}$};
\draw (1.8,-0.5) node  {$\overline{\pi}_{i+2,j+1}$};
\draw (5.1,-0.5) node  {$\overline{\pi}_{i,j+1}$};
\draw (7.8,-0.5) node  {$\overline{\pi}_{i+1,j}$};
\draw (3.7,-0.5) node  {$\oplus$};
\draw (6.4,-0.5) node  {$\oplus$};
\draw (4.5,-1.5) node  {$\overline{\pi}_{i+1,j+1}$};
\end{tikzpicture}
\end{center}
\end{minipage}

\vspace{2mm}
where:
\newline $\overline{\pi}_{0,1}$ is finite dimensional. 
\newline $\overline{\pi}_{0,j}$ for $2\leq j\leq 2n$, $\overline{\pi}_{1,2}$, and ${\pi}_{0}$ have Gelfand-Kirillov dimension $2n+1$. 
\newline $\overline{\pi}_{i,j}$ for $1\leq i\leq n$ and $i+1\leq j\leq 2n$, and $\pi_{l}$ for $1\leq l\leq n$ have Gelfand-Kirillov dimension $4n-1$. 
\end{Thm}


\section{Algorithm for finding the residue representations}

\label{section Algo}
In \cite{ROBY2021}, the residue representations have been completely determined in terms of their Langlands parameters under the assumption that the resonances arise from the poles the trivial Plancherel density. In this section, we present an algorithm which allows us to compute the Langlands parameters of each residue representation, the Gelfand-Kirillov dimension of the space of these representations and their wave front set under no restriction on $\tau$. 
The input of the algorithm is the highest weight of $\tau$. 

The majority of tools we are using in this part comes from \cite{CollinRROLG}. 
The idea which we borrow from this book is to reduce the study of the composition series of every principal series representation to the case of principal series representations with trivial infinitesimal character using the translation functors \eqref{eq Def up/down functors}. 

We prove first two facts which will simplify the algorithm. 

\begin{Prop}\label{Prop reg inf ch}
Let $\sigma$ be a $M$-type of $\tau$. We denote by $\lambda_k$, with $k\in \Z_+$ the (negative) poles of the Plancherel density $p_\sigma$ associated with $\sigma$. 
Then the principal series $\Hi_{\lambda_k}^\sigma$ has always a regular infinitesimal character $\gamma_{\lambda_k}^\sigma$. 
\end{Prop}

\begin{proof}
Let us first recall the precise formula of the polynomial part $P_\sigma$ of the Plancherel density $p_\sigma$. This formula, for real rank one groups, can be found in \cite[page 543]{Knapp&Stein}:
\begin{equation*}
    P_\sigma(z) = \prod_{\eta \in \Pi^+}\langle \eta,z\alpha +i(\mu_\sigma+\rho_\m)\rangle~,
\end{equation*}
which can be decomposed in 
\begin{equation*}
    P_\sigma(z) = \prod_{\eta \in \Pi^+\setminus\Pi_\m^+}\langle \eta,z\alpha +i(\mu_\sigma+\rho_\m)\rangle ~\times~\prod_{\eta \in \Pi_\m^+}\langle \eta,i(\mu_\sigma+\rho_\m)\rangle~.
\end{equation*}
Here just the first product is interesting to us. Recall also that we do not take into account multiplicative constants. Recall that $\Pi^+\setminus\Pi_\m^+$ is the set of positive roots not vanishing on $\Aa$.  
Let us recall what happens for each case.
\begin{center}
    \begin{tabular}{|Sc|Sc|}
\hline
    G & $\Pi^+\setminus\Pi^+_\m$ \\
     \hline
     $\Spin(2n,1)$ & $\varepsilon_1, ~~\varepsilon_1\pm\varepsilon_j , ~2\leq j \leq n$ \\
     \hline
     $\SU(n,1)$ & $\varepsilon_1-\varepsilon_{n+1}, ~~\begin{array}{cc}\varepsilon_1-\varepsilon_j ,& ~2\leq j \leq n,\\\varepsilon_j-\varepsilon_{n+1} ,& ~2\leq j \leq n\end{array}$\\
     \hline 
     $\Sp(n,1)$ &$2 \varepsilon_1$, $2 \varepsilon_2$,$\begin{array}{cc}
         \varepsilon_1\pm\varepsilon_j,& ~2\leq j \leq n+1  \\
          \varepsilon_2\pm\varepsilon_j,& ~2\leq j \leq n+1
     \end{array}$ \\
     \hline 
     $\F_{4(-20)}$ & $\begin{array}{c}\varepsilon_1,~\varepsilon_1\pm\varepsilon_j, ~2\leq j \leq 4\\ \frac12 (\varepsilon_1\pm\varepsilon_2\pm\varepsilon_3\pm\varepsilon_4) \end{array}$\\
     \hline
\end{tabular}
\end{center}
Decomposing each root in term of the fundamental weights, we get, up to a constant:
\begin{equation*}
    P_\sigma(z) = \prod_{\eta \in \Pi^+\setminus\Pi_\m^+}\Big(\langle \eta,z\alpha\rangle^2 +\langle \eta,\mu_\sigma + \rho_\m\rangle^2\Big) ~.
\end{equation*}
Thus the zeros of this product are in the set
\begin{equation*}
    \left\{\pm\langle \eta,\mu_\sigma+\rho_\m\rangle ~\Big|~\eta\in\Pi_\m^+\right\}~.
\end{equation*}
If $\lambda\alpha$ is a pole of the Plancherel formula, then $\lambda$ is not in this set. This implies that $\gamma_{\lambda\alpha}^\sigma$ is regular. We prove this fact for $\SU(n,1)$, the other cases being similar.
We recall that for $G=\SU(n,1)$
\begin{align*}
    \gamma_{\lambda\alpha}^\sigma &= \lambda(\varepsilon_1-\varepsilon_{n+1}) + \mu_\sigma+\rho_\m \\&= (\lambda+\langle \varepsilon_1,\mu_\sigma+\rho_\m\rangle)\varepsilon_1+ (-\lambda+\langle \varepsilon_{n+1},\mu_\sigma+\rho_\m\rangle)\varepsilon_{n+1}) + \sum_{i=2}^n \langle \varepsilon_i,\mu_\sigma+\rho_\m\rangle \varepsilon_i~.
\end{align*}
So, $\gamma_{\lambda\alpha}^\sigma$ is regular if and only if $\lambda \ne \langle \varepsilon_i-\varepsilon_1,\mu_\sigma+\rho_\m\rangle$ and $\lambda \ne \langle \varepsilon_{n+1}-\varepsilon_i,\mu_\sigma+\rho_\m\rangle$ for all $2\leq i \leq n$. 
This is exactly the set of zeros of the polynomial part of the Plancherel density given by the numbers 
\begin{equation*}
    0, ~~\pm\langle\varepsilon_1-\varepsilon_j,\mu_\sigma+\rho_\m\rangle ,~~ \pm\langle\varepsilon_j-\varepsilon_{n+1},\mu_\sigma+\rho_\m\rangle , ~~ \text{for } 2\leq j \leq n~.
\end{equation*}
\end{proof}

The following well-known lemma assures that the study of the subquotients of the principal series representations with the trivial infinitesimal characters is sufficient to identify the subquotient of any principal series representation. Not having found a reference, we include a proof. 

\begin{Lemma}\label{Lemma phi image}
Let $\gamma_1$ and $\gamma_2$ two regular characters such that the unique irreducible quotient $\overline{\pi}(\gamma_1)$ is a component of $\pi(\gamma_2)$. Then $\pi(\gamma_1)$ and $\pi(\gamma_2)$ have the same infinitesimal character and we can therefore fix $w\in W_\C$ such that $w\cdot\gamma_1 = \gamma_2$. 
Moreover for all $\mu\in \Lambda$ \begin{equation}
    \Phi_{w(\gamma_1+\mu)}^{w\gamma_1}(\overline{\pi}(\gamma_1)) = \Phi_{\gamma_1+\mu}^{\gamma_1}(\overline{\pi}(\gamma_1))~.
\end{equation}
\end{Lemma}



\begin{proof}
First \cite[Remark 4.3.3.ii]{CollinRROLG} assures that $\Psi^{\gamma_1+\mu}_{\gamma_1}(\overline{\pi}(\gamma_1+\mu)) = \overline{\pi}(\gamma_1)$ and $\Psi^{\gamma_2+w\mu}_{\gamma_2}(\overline{\pi}(\gamma_2+w\mu)) = \overline{\pi}(\gamma_2)$. 
Moreover, by \cite[Lemma 7.3.1 page 462]{VoganRRRLG}, $\Psi_{w(\gamma_1+\mu)}^{w\gamma_1} = \Psi_{\gamma_1+\mu}^{\gamma_1}$. 
As $\overline{\pi}(\gamma_1)$ is irreducible, we have 
\begin{equation*}
    \begin{array}{ll}
        \Phi_{\gamma_1+\mu}^{\gamma_1} \Psi^{\gamma_1+\mu}_{\gamma_1} (\overline{\pi}(\gamma_1+\mu)) &= \overline{\pi}(\gamma_1+\mu)~,\\
        \Phi_{\gamma_2+w\mu}^{\gamma_2} \Psi^{\gamma_2+w\mu}_{\gamma_2} (\overline{\pi}(\gamma_1+\mu)) &= \overline{\pi}(\gamma_1+\mu)~.
\end{array}
\end{equation*}
Putting everything together, we obtain
\begin{equation*}
    \Phi_{\gamma_1+\mu}^{\gamma_1}(\overline{\pi}(\gamma_1)) = \overline{\pi}(\gamma_1+\mu)=
    \Phi_{\gamma_2+w\mu}^{\gamma_2}(\overline{\pi}(\gamma_1)) ~,
\end{equation*}
which proves that $$\Phi_{\gamma_2+w\mu}^{\gamma_2}(\overline{\pi}(\gamma_1)) = \Phi_{\gamma_1+\mu}^{\gamma_1}(\overline{\pi}(\gamma_1))~.$$
\end{proof}


Now we can describe the algorithm to find each residue representation, starting from the representation $\tau\in \hat{K}$ which determines the homogeneous vector bundle $E_\tau$. 
From now on, consider a fixed representation $\tau \in \hat{K}$, with a known highest weight $\mu_\tau$.

\vspace{20mm}
\begin{Thm}\label{Thm algo}
For every $\sigma \in \hat{M}(\tau)$, the residue representation can be found by the following algorithm. 
\begin{enumerate}
    \item Use \cite{Bald1} to compute all the highest weights $\mu_\delta$ for $\delta \in \hat{M}(\tau)$.
    \item Use \cite{Mia} or \cite{ROBY2021} to compute the poles of the Plancherel density associated with $\sigma$. We denote them by $\pm \lambda_k^\sigma \in \frac12\Z$, $k\in \Z_+$, where $\lambda_k^\sigma <0$. For each pole, we define the representation $\E^\sigma_k$ as in section \ref{section residue repr}.
    \item Compute the infinitesimal character of $\Hi_k^\sigma := \Hi^\sigma_{\lambda^\sigma_k}$, which is given
    \begin{equation}\label{eq inf ch princ series}
        \gamma_k^\sigma = \lambda^\sigma_k + \mu_\sigma + \rho_\m~ \in \h^*_\C ~,
    \end{equation}
    where we recall that $\rho_\m$ is the half sum of the positive roots that are $0$ on $\Aa$. Of course, the values of these elements has to be written in terms of fundamental weights of $\Delta(\g_\C,\h_\C)$. 
    \item By one of Theorem 3,4,5 (according to the group $G$) in section \ref{section Col's Thm used}, find the trivial regular character $\gamma_{i,j}$ which corresponds to $\gamma_k^\sigma$. It may depend on $k$. The correspondent pair $(i,j)$ singles out a highest weight $\mu$ such that $\gamma_{i,j} +\mu = \gamma_k^\sigma$. The infinitesimal character $\gamma_{i,j}$ is also the unique trivial inifinitesimal character which induces the same positive root system as $\gamma_k^\sigma$. They are both strictly dominant with respect to this root system. 
    \item By the same theorem, find the decomposition of $\pi_{i,j}$.
    \begin{enumerate}
        \item For each $\delta \in \hat{M}(\tau)$, verify if there exists $w\in W(\g_\C,\h_\C)$ and $\nu\in \Aa_\C^*$ such that $w\cdot\gamma_k^\sigma=\gamma_\nu^\delta$, the infinitesimal character of $\Hi_{\nu}^\delta$. Denote each $\delta$ by $\delta_w$.
        \item By one of Theorem 3,4,5 in section \ref{section Col's Thm used}, find the $\gamma_{l,m} =: \gamma_{\delta_w}$ corresponding to each $\gamma_\nu^{\delta_w}$. 
        \item Write all the compositions series of every $ \pi_{\delta_w}:=\pi_{l,m}$ corresponding to the $\gamma_{\delta_w}$. Suppose we are in multiplicity free case. Then there exists just one subquotient $\overline{\pi}$ which 
        \begin{enumerate}
            \item is in every $\pi_{\delta_w}$ ($\pi_{i,j}$ included),
            \item is the maximal subquotient of one $\pi_{\delta_w}$,
            \item does not appear in any other $\pi_{l,m}\ne \pi_{\delta_w}$, for one $\delta_w$.
        \end{enumerate}
        If $m(\tau|_M,\delta_w) >1$, one can have $m(\tau|_M,\sigma)$ of subquotients verifying the same conditions $(ii)$ and $(iii)$. The condition $(i)$ becomes 
        \newline$(i')$ Every $\pi_{\delta_w}$ contains at least one of these $\overline{\pi}$ subquotients. \newline
        $\E_k$ is the sum of these subquotients.
         \item If one cannot find a suitable $\overline{\pi}$, there are $\pi_l$ (discrete series representations) in $\pi_{i,j}$. If there is one, it is in $\E_k$. If there is not just one, you have to decide which one is in $\E_k$. One can use \cite[Theorem 1]{Parthasarathy} to decide which one. 
    \end{enumerate}
    \item We thus get the residue representation $\E^\sigma_k$ as a sum of Langlands quotient(s). The Gelfand-Kirillov dimension is given by Theorems 3,4,5 in section \ref{section Col's Thm used}. The wave front set can be then computed thanks to \cite[Section 5]{ROBY2021}. 
\end{enumerate}
\end{Thm}

\begin{proof}[Why the algorithm works]
\begin{enumerate}
\item This is just a direct computation. Knowing the highest of any $\sigma\in \hat{M}(\tau)$ is essential. 
\item This step is also a direct computation. 
\item The reader can see for example \cite[Proposition 8.22 and Lemma 12.28]{Kna1} for the formula \eqref{eq inf ch princ series}.
\item Thanks to Proposition \ref{Prop reg inf ch}, we know that $\gamma_k^\sigma$ is regular. By Theorem \ref{Thm Col: red inf ch}, $\gamma_k^\sigma$ corresponds to a trivial infinitesimal character $\gamma_{i,j}$. If we denote by $\Pi_{i,j}^+$ the positive Weyl chamber for which $\gamma_{i,j}$ and $\gamma_k^\sigma$ are strictly dominant, then there exists a highest weight $\mu$ with respect to $\Pi_{i,j}^+$ such that $\gamma_{i,j} + \mu = \gamma_k^\sigma$. 
Thanks to Theorems 3,4,5 in section \ref{section about Col}, we have the composition series of $\pi_{i,j}$ and thus of $\Hi_k^\sigma$.
\item \cite[Lemma 4.1]{ROBY2021} proves that $\E^\sigma_k$ is composed by the sum of the components containing the $K$-type $\tau$. The goal is now to get this information explicitely. 
These components are the $\overline{\pi}^{\sigma,k}_{l,m}:= \Phi_{\gamma_k^\sigma}^{\gamma_{i,j}}(\overline{\pi}_{l,m})$, for each $\overline{\pi}_{l,m}$ in $\pi_{i,j}$. Suppose we are in multiplicity free. There is then just one component containing $\tau$. 
It must be the maximal quotient of a principal series induced by a $\sigma\in \hat{M}(\tau)$ or a discrete series representation. 

As in the Theorem, for each $\delta \in \hat{M}(\tau)$, verify if there exists $w\in W(\g_\C,\h_\C)$ and $\nu\in \Aa_\C^*$ such that $w\cdot\gamma_k^\sigma=\gamma_\nu^\delta$, the infinitesimal character of $\Hi_{\nu}^\delta$. Denote each $\delta$ by $\delta_w$.
Lemma \ref{Lemma phi image} assures us that, two subquotients, contained in principal series having the same infinitesimal character and corresponding to the same $\overline\pi_{l,m}$ after restriction to the trivial infinitesimal character case, are the same. We reduced ourselves to find the overlapping subquotient in the trivial infinitesimal character case.
\newline Thus we compute the trivial infinitesimal character corresponding to each $\gamma_\nu^\delta$ ($\gamma_k^\sigma$ included) and deduce a composition series decomposition. If $m(\tau|_M,\delta_w) =1$ for any $\delta_w$, then we have just to find one overlapping subquotient $\overline\pi_{l,m}$. Recall that this subquotient has to be in every composition series but not in other principal series representations. If no subquotients succeed to verifying these conditions, $\tau$ appears in discrete series representations. One can use \cite[Theorem 1]{Parthasarathy} to decide. 
If $m(\tau|_M,\delta_w) \ne1$, at least for one $\delta_w$, the situation is more complicated, but we have enough information to conclude with the same spirit as before. See the case of $\tau_{1,4,4}$ in section \ref{section pforms Sp} as an example. 

This gives the components where $\tau$ is (or not) and thus the residue representations $\E_k^\sigma$ in terms of (sum of) Langlands quotient(s). If not, a direct application of \cite[Theorem 1]{Parthasarathy} is enough to conclude. 
\item This is a direct consequence of the Theorems and sections cited. 
\end{enumerate}
\end{proof}


\section{Case of the Laplacian of the $p$-forms}
\label{section pforms}

We denote by $\K$ one of $~\R, ~\C,$ or $\Hh$.
Let $E_p := \Lambda^p H^n(\K)$ be the space of $p$-forms over the hyperbolic spaces $G/K = H^n(\K)$. This is the space of sections of the homogeneous vector bundle over $G/K$ associated with the representation $\tau_p : = \Lambda^p \Ad^\ast$, where $\Ad^\ast$ denotes the coadjoint representation of $K$ on $\p_\C^\ast$. This representation is very well described in \cite{CampHigu-pforms,PedTh,PedonpformsC,PedonpformsH,PedCar}, where one can find harmonic analysis on $E_p$. 
In this paper, we do not consider the case $p=0$, where the $K$-type $\tau_0$ is the trivial representation of $K$, and the sections of the bundle $E_{0}$ are just functions on $G/K$. This case is considered and completely described in \cite{HilgPasq}. 

We recall some facts about the structure of $\tau_p$ at the beginning of each section. Then we apply the algorithm described in section \ref{section Algo} and  find the residue representations associated to the resonances of the Laplace operator acting on compactly supported smooth sections of $E_p$, $p\in [1,n]\cap\Z^*_+$. 

\subsection{Real case: $G = \Spin(2n,1)$, $n\geq 2$}

In this case, $K = \Spin(2n)$ and $M=\Spin(2n-1)$. We are not considering $G = \Spin(2n+1,1)$, because there are no resonance in this case (see \cite{ROBY2021}) . 

\subsubsection{Decomposition of the representations}

One can use the branching rules in \cite{Bald1} as in \cite{PedTh}. 

If $p < n$, then $\tau_p$ is irreducible on $K$ and has highest weight $\mu_p = \varepsilon_2 + \cdots + \varepsilon_{p+1}$. When we restrict $\tau_p$ to $M$, we obtain the decomposition $$\tau_p | _M = \sigma_p \oplus \sigma_{p-1}$$ where $\sigma_p$ and $\sigma_{p-1}$ have highest weights $\mu_p$ and $\mu_{p-1}$ respectively. The representation $\tau_p$ decomposes as follows: $$\Lambda^p \C^n = \left( \Lambda^{p} \C^{n-1}\right) \oplus e_1 \wedge \left( \Lambda^{p-1} \C^{n-1}\right)~.$$

The representation $\tau_n$ has two equivalent irreducible subrepresentations. More specifically, we have $$\Lambda^n \C^n = \Lambda^n_+ \C^n \oplus \Lambda^n_- \C^n~.$$
We denote these two irreducible subrepresentations respectively by $\tau_n^-$ and $\tau_n^+$, of respective highest weights: \begin{center}
    $\mu_n^- = \varepsilon_2 + \cdots + \varepsilon_{n} - \varepsilon_{n+1}$ ~~and~~ $\mu_n^+ = \varepsilon_1 + \cdots + \varepsilon_{n+1}$. 
\end{center}
The restriction of any of them to $M$ is irreducible and is equal to the same representation $$\tau_n^- \sim \tau_n^+\sim \sigma_{n-1}\sim \sigma_{n}$$ with highest weight
$\mu_{n-1} = \varepsilon_2 + \cdots + \varepsilon_{n}$.

Since $\tau_p \sim \tau_{2n-p}$, there are no other cases. 

\subsubsection{Poles of the Plancherel density}

This is a direct computation done in \cite[Proposition 3.1 and Appendix A.1]{ROBY2021} using \cite{Mia}. These singularities have been found first by Pedon \cite[p.110]{PedTh}.

$\left.\begin{minipage}{0.35\textwidth}
\definecolor{light-gray}{gray}{0.75}
\begin{center}
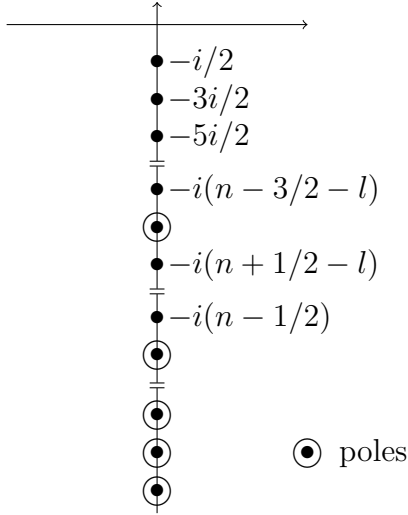

\begin{tikzpicture}[scale=1]
\draw [->](-2,0) to (2,0);
\draw [<-|](0,0.3) to (0,-1.825);
\draw [|-|](0,-1.875) to (0,-3.525);
\draw [|-|](0,-3.575) to (0,-4.775);
\draw [|-](0,-4.825) to (0,-6.5);
\draw (0,-3.9) node[right] {$-i(n-1/2)$} node{$\bullet$};
\draw (0,-3.2) node[right] {$-i(n+1/2-l)$} node{$\bullet$};
\draw (0,-2.7)  node{$\bullet$} node{$\bigcirc$};
\draw (0,-2.2) node[right] {$-i(n-3/2-l)$} node{$\bullet$};
\draw (0,-1.5) node[right] {$-5i/2$} node{$\bullet$};
\draw (0,-1) node[right] {$-3i/2$} node{$\bullet$};
\draw (0,-0.5) node[right] {$-i/2$} node{$\bullet$};
\draw (0,-4.4)  node{$\bullet$} node{$\bigcirc$};
\draw (0,-5.2) node{$\bullet$} node{$\bigcirc$};
\draw (0,-5.7)  node{$\bullet$} node{$\bigcirc$};
\draw (0,-6.2)  node{$\bullet$} node{$\bigcirc$};
\draw (2,-5.7)  node{$\bullet$} node{$\bigcirc$} node[right]{~~poles};
\end{tikzpicture}
\captionof{figure}{Singularities of the Plancherel densities for $\sigma_l$}\label{singularities}
\end{center}
\end{minipage}
\right|$
\begin{minipage}{0.05\textwidth}
\end{minipage}
\begin{minipage}{0.55\textwidth}
\begin{Lemma}\label{Lemma singularities p_sigma p forms}
The singularities of the Plancherel density corresponding to $\sigma_l$ are located at
\begin{equation}\label{eq singularities p_sigma p forms}
    \left\{\pm i(n-1/2-l), ~\pm i(\rho_\alpha+k) ~\text{ for } k\in\Z^\times_+\right\}
\end{equation}
For $k\in\Z^\times_+$, set
\begin{equation}
    \lambda^l_k = -i(\rho_\alpha+k) ~~~\text{ and } ~~~~ \lambda^l_0 = -i(n-1/2-l).
\end{equation}
\end{Lemma}
\end{minipage}

\begin{Rem}[Case $p = n$]
Recall that in this case, the vector bundle $E_n$ decomposes as a direct sum of two, one for each $K$ type $\tau^\pm_n$ in $\tau_n$. The singularities of $p_{\sigma_{n-1}}$ are given by formula \eqref{eq singularities p_sigma p forms} for $l = n-1$ and $\tau_p = \tau^\pm_n$. But as the two vector bundles are different, we will see that the same poles of the same Plancherel density can correspond to different residue representations.
\end{Rem}

\subsubsection{Residue representations}

We now have all the ingredients to study the residue representation $\E^{l}_k :=\E^{\sigma_l}_k$ described in section \ref{section residue repr}.  
The residue representations in \eqref{eq residue representation} become 
\begin{equation}
    \E_k^l := \{\varphi_{p}^{l,k} \ast f ~|~ f \in  C_c^\infty(G,\tau_p)\}~,
\end{equation}
for $l= p$ or $p-1$ and $k \in \Z_+$. Recall that $p=n$ is equivalent to $l= n-1$. To simplify notation, we set $\varphi_{p}^{l,k}:=\varphi_{\tau_p}^{\sigma_l,\lambda^l_k\alpha}$ for the spherical function and we denote the principal series representation $\Ind_{MAN}^G(\sigma_l\otimes e^{i\lambda^l_k\alpha}\otimes 1)$, in which $\E_k^l $ is embedded, by $(\Hi^l_k, \pi^l_k) :=(\Hi^{\sigma_l}_{\lambda^l_k\alpha}, \pi^{\sigma_l}_{\lambda^l_k\alpha})~.$

\begin{Prop}\label{Prop res repr pforms - Rcase}
The residue representation $\E_k^l$ is always irreducible and 
\newline
\begin{tabular}{ll}
    \textbullet~ finite dimensional with Langlands parameters $(MA,\sigma_l,\pm i\lambda_k^l\alpha)$, & if $k\in \Z_+^\times$\\
    \textbullet~ infinite dimensional, with Gelfand-Kirillov dimension $2n-1$ \\and Langlands parameters $(MA,\sigma_p,\pm i\lambda_0^p\alpha)$, & if $k=0$ and $p\ne n$\\
    \textbullet~ is the discrete series representation with Harish-Chandra parameter $\gamma_0$, & if $k=0$ in the $\tau_n^+$ case\\
    \textbullet~ is the discrete series representation with Harish-Chandra parameter $\gamma_1$, & if $k=0$ in the $\tau_n^-$ case
\end{tabular}
\end{Prop}

\begin{proof}
we prove this proposition to illustrate our algorithm, even if for this special case simpler direct proof are possible. 
\begin{enumerate}
    \item and (2) were already proved before the proposition.
    \setcounter{enumi}{2}
    \item Here $\varepsilon_1$ is the real root. We get then, for $l=p,p-1$, and $k\ne0$,
    $$\gamma_k^{\sigma_l} = \left(n-\frac12+k\right)\varepsilon_1 + \sum_{i=2}^{l+1}\left(n-i+\frac32\right)\varepsilon_i+\sum_{i=l+2}^{n}\left(n-i+\frac12\right)\varepsilon_i~,$$ and 
    $$\gamma_0^{\sigma_l} = \left(n-\frac12-l\right)\varepsilon_1 + \sum_{i=2}^{l+1}\left(n-i+\frac32\right)\varepsilon_i+\sum_{i=l+2}^{n}\left(n-i+\frac12\right)\varepsilon_i~.$$
    \item If $k\ne0$, then $\gamma_k^{\sigma_l}$ is associated with $\gamma_{0,1}$ in Theorem 3. The components in $\pi_{0,1}$ are $\overline{\pi}_{0,1}$ and $\overline{\pi}_{0,2}$.  \newline
    If $k=0$ and $p\ne n$, then $\gamma_0^{\sigma_l}$ is associated with $\gamma_{0,l+1}$ in Theorem 3. The components in $\pi_{0,l+1}$ are $\overline{\pi}_{0,l+1}$ and $\overline{\pi}_{0,l+2}$.\newline
    If $k=0$ and $p= n$, then $\gamma_0^{\sigma_{n-1}}$ is associated with $\gamma_{0,n}$ in Theorem 3. The components in $\pi_{0,n}$ are $\overline{\pi}_{0,n}$, $\overline{\pi}_{0}$ and $\overline{\pi}_{1}$.
    \item If $k\ne0$, let $w\in W(\g_\C,\h_\C)$ such that $w\cdot\gamma_{0,1} = \gamma_{0,2}$. Then 
    $$w\cdot\gamma_k^{\sigma_l} = \left(n-\frac12\right)\varepsilon_1 + \left(n-\frac12+k\right)\varepsilon_2 + \sum_{i=3}^{l+1}\left(n-i+\frac32\right)\varepsilon_i+\sum_{i=l+2}^{n}\left(n-i+\frac12\right)\varepsilon_i~.$$
    This is the infinitesimal character of the principal series induced by $\left(n-\frac12\right)\alpha$ and the $M$-type with highest weight
    $(k+1)\varepsilon_2 + \sum_{i=3}^{l+1}\varepsilon_i$.
    As $k+1>1$, this $M$-type is not contained $\tau_p|_M$. Thus $\tau_p$ is in $\overline{\pi}_{0,1}$.\newline
    If $k=0$ and $p\ne n$, let $w_0, w_1\in W(\g_\C,\h_\C)$ such that $w_0\cdot\gamma_{0,p+1} = \gamma_{0,p+2}$ and $w_1\cdot\gamma_{0,p} = \gamma_{0,p+1}$. Then 
    $$w_0\cdot\gamma_0^{\sigma_p} = \left(n-p-\frac32\right)\varepsilon_{1}+ \sum_{i=2}^{p+2}\left(n-i+\frac32\right)\varepsilon_i+\sum_{i=p+3}^{n}\left(n-i+\frac12\right)\varepsilon_i~.$$
    This is the infinitesimal character of the principal series induced by $\left(n-p-\frac32\right)\alpha$ and the $M$-type $\sigma_{p+1}$. As $\tau_p|_M$ does not contain $\sigma_{p+1}$, $\tau_p$ is in $\overline{\pi}_{0,p+1}$. Moreover 
    $$w_1\cdot\gamma_0^{\sigma_p} = \left(n-p-\frac12\right)\varepsilon_{1}+ \sum_{i=2}^{p+1}\left(n-i+\frac32\right)\varepsilon_i+\sum_{i=p+2}^{n}\left(n-i+\frac12\right)\varepsilon_i~.$$ 
    This is the infinitesimal character of the principal series induced by $\left(n-p-\frac12\right)\alpha$ and the $M$-type $\sigma_{p}$. As $\tau_p|_M$ does contain $\sigma_{p}$, $\tau_p$ is in $\overline{\pi}_{0,p+1}$. 
    \newline
    If $k=0$ and $p=n$, let $w\in W(\g_\C,\h_\C)$ such that $w\cdot\gamma_{0,n} = \gamma_{0,n-1}$. 
    $$w\cdot\gamma_0^{\sigma_{n-1}} = \frac32\varepsilon_1 + \sum_{i=2}^{n-1}\left(n-i+\frac32\right)\varepsilon_i+\frac12\varepsilon_n~.$$
    This is the infinitesimal character of the principal series induced by $\frac32\alpha$ and the $M$-type $\sigma_{n-2}$. As $\tau^\pm_n|_M$ both do not contain $\sigma_{n-2}$, $\tau^\pm_n$ are $K$-types of $\pi_{0}$ or $\pi_{1}$. By direct computation and \cite[Theorem 1]{Parthasarathy}, one gets that $\tau^+_n$ is a $K$-type of $\pi_{0}$ and $\tau^-_n$ is a $K$-type of $\pi_{1}$.
    \item The Gelfand-Kirillov dimensions are given in Theorem 3. 
\end{enumerate}
\end{proof}

\begin{Cor}
The wave front sets of the infinite dimensional representations in Proposition \ref{Prop res repr pforms - Rcase} are all equal to the nilpotent orbit generated by $\g_\alpha$.
\end{Cor}

As there is only one nonzero wave front set (see \cite[Theorem 2]{ROBY2021}), the wave front set of infinite dimensional representations is always the same.

\subsection{Complex case : $G = \SU(n,1)$, $n\geq 2$}

We recall that $K = \Super(\U(n)\times\U(1))$ and $M=\Super(\U(n-1)\times\U(1))$.

\subsubsection*{Decomposition of the representations}
One can use the branching rules in \cite{Bald1} as in \cite{PedonpformsC}.
The representation $\tau_p$ decomposes into $K$-types $\tau_{q-k,r-k}$ as follows
\begin{equation}
    \tau_p = \bigoplus_{q+r=p}\bigoplus^{\min{(q,r)}}_{k=0} \tau_{q-k,r-k}~,
\end{equation}
where $\tau_{a,b}$ has highest weight 
\begin{equation}
    \mu_{a,b} = \sum_{i=1}^b\varepsilon_i - \sum_{k=n-a+1}^n \varepsilon_i + (a-b)\varepsilon_{n+1}~.
\end{equation}
Many $K$-types appear in the decomposition, contrary to the real case. Note that $0\leq a,b\leq p\leq n$ and $a+b \leq p\leq n$.  On $M$ we have the following decomposition of each $K$-type
\begin{equation}
    \tau_{q,r}|_M = \bigoplus_{\substack{l~=~q,q-1\\m~=~r,r-1}} \sigma_{l,m}~,
\end{equation}
where $\sigma_{l,m}$ has highest weight
\begin{equation}
    \mu_{l,m}=\sum_{i=2}^{m+1}\varepsilon_i - \sum_{k=n-l+1}^n \varepsilon_i + \frac{l-m}2 (\varepsilon_{1}+\varepsilon_{n+1})~.
\end{equation}
In the decomposition above, $\sigma_{l,m} =0$ if $\min{(l,m)} <0$ or $\max{(l,m)} > n-1$.

\subsubsection{Poles of the Plancherel density}

This is a direct computation done in \cite[Proposition 3.1 and Appendix A.2]{ROBY2021} using \cite{Mia}. These singularities have been found first by Pedon \cite{PedonpformsC}.

$\left.
\begin{minipage}{0.35\textwidth}
\definecolor{light-gray}{gray}{0.75}
\begin{center}
\begin{tikzpicture}[scale=1]
\draw [->](-2,0) to (2,0);
\draw [<-|](0,0.3) to (0,-1.825);
\draw [|-|](0,-1.875) to (0,-3.525);
\draw [|-|](0,-3.575) to (0,-4.775);
\draw [|-](0,-4.825) to (0,-6.5);
\draw (0,-3.9) node[right] {$-i\left(\frac{n+|m-l|}2+1\right)$} node{$\bullet$};
\draw (0,-3.2) node[right] {$-i\left(\frac{n-m-l}2+1\right)$} node{$\bullet$};
\draw (0,-2.7)  node{$\bullet$} node{$\bigcirc$};
\draw (0,-2.2) node[right] {$-i\left(\frac{n-m-l}2-1\right)$} node{$\bullet$};
\draw (0,-1.5)  node{$\bullet$};
\draw (0,-1) node{$\bullet$};
\draw (0,-0.5)  node{$\bullet$};
\draw (0,-4.4)  node{$\bullet$} node{$\bigcirc$};
\draw (0,-5.2) node{$\bullet$} node{$\bigcirc$};
\draw (0,-5.7)  node{$\bullet$} node{$\bigcirc$};
\draw (0,-6.2)  node{$\bullet$} node{$\bigcirc$};
\draw (2,-5.7)  node{$\bullet$} node{$\bigcirc$} node[right]{~~poles};
\end{tikzpicture}
\captionof{figure}{Singularities of the Plancherel densities for $\sigma_l$}\label{singularities}
\end{center}
\end{minipage}
\right|$
\begin{minipage}{0.02\textwidth}
\hspace{5mm}
\end{minipage}
\begin{minipage}{0.55\textwidth}
\begin{Lemma}\label{Lemma singularities p_sigma p forms}
The singularities of the Plancherel density corresponding to $\sigma_{l,m}$ are located at
\begin{equation}\label{eq singularities p_sigma p forms}
    \left\{\pm i~~\frac{n-m-l}2, ~\pm i\left(\frac{n+|m-l|}2+k\right) ~\text{ for } k\in\Z^\times_+\right\}
\end{equation}
For $k\in\Z^\times_+$, set
\begin{equation}
    \lambda^{l,m}_k = -i\left(\frac{n+|m-l|}2+k\right) ~~~\text{ and } ~~~~ \lambda^{l,m}_0 = -i~~\frac{n-m-l}2.
\end{equation}
\end{Lemma}
\end{minipage}

\subsubsection{Residue representations}
We now have all the ingredients to study the residue representations $\E^{l,m}_k :=\E^{\sigma_{l,m}}_k$ described in section \ref{section residue repr}.  
The residue representations in \eqref{eq residue representation} become 
\begin{equation}
\E_k^{l,m} := \{\varphi_{p,q}^{l,m,k} \ast f ~|~ f \in  C_c^\infty(G,\tau_{p,q})\}~,
\end{equation}
for $l= p$ or $p-1$, $m= q$ or $q-1$ and $k \in \Z_+$. To simplify notation, we set $\varphi_{p,q}^{l,m,k}:=\varphi_{\tau_{p,q}}^{\sigma_{l,m},\lambda^{l,m}_k\alpha}$ for the spherical function and we denote the principal series representations $\Ind_{MAN}^G(\sigma_{l,m}\otimes e^{i\lambda_k^{l,m}}\otimes 1)$, in which $\E^{\sigma_{l,m}}_k$ is embedded, by $\Hi_k^{l,m}:=\Hi^{\sigma_{l,m}}_{\lambda_k^{l,m}}~.$

\begin{Prop}\label{Prop res repr pforms - Ccase}
The residue representation $\E_k$ is always irreducible and
\newline

\begin{itemize}
    \item \underline{if $k\ne0$}: finite dimensional with Langlands parameters $(MA,\sigma_{l,m},\pm i\lambda_k^{l,m}\alpha)$,
    \item \underline{if  $k=0$, $q+r\ne n$ and $qr=0$}: infinite dimensional with Gelfand-Kirillov dimension $n$ and Langlands parameters $(MA,\sigma_{q,r},\pm i\lambda_0^{q,r}\alpha)$,
    \item \underline{if $k=0$, $q+r\ne n$ and $qr\ne0$}: infinite dimensional with Gelfand-Kirillov dimension $2n-1$ and Langlands parameters $(MA,\sigma_{q,r},\pm i\lambda_0^{q,r}\alpha)$,
    \item \underline{if $k=0$ and $q+r = n$}: the discrete series representation with Harish-Chandra parameter $\gamma_r$.
\end{itemize}

     
\end{Prop}

\begin{proof}\textcolor{white}{l}
\begin{enumerate}
    \item and (2) were already proved before the proposition.
    \setcounter{enumi}{2}
    \item The infinitesimal character of $\Hi_k^{l,m}$ is given by 
\begin{align*}
    \gamma_k^{l,m} &= i\lambda^{l,m}_k(\varepsilon_1-\varepsilon_{n+1}) + \mu_{l,m} + \rho_m \\
                 &= \left\{\begin{array}{cc}
                     \left(\frac{n+|m-l|}2+k\right) & \text{ if } k\in \Z_+^\times \\
                      \frac{n-m-l}2 & \text{ if } k= 0 
                 \end{array}\right\}(\varepsilon_1-\varepsilon_{n+1}) + \sum_{i=2}^{m+1}\varepsilon_i - \sum_{k=n-l+1}^n \varepsilon_i + \frac{l-m}2 (\varepsilon_{1}+\varepsilon_{n+1}) \\&+ \frac12 \sum_{i=2}^n (n-2i+2)\varepsilon_i  \\
                 &= \left\{\begin{array}{cc}
                     \left(\frac{n+|m-l|+l-m}2+k\right)\varepsilon_1 - \left(\frac{n+|m-l|-l+m}2+k\right)\varepsilon_{n+1} & \text{ if } k\in \Z_+^\times \\
                      \left(\frac{n}2-m\right) \varepsilon_1 - \left(\frac{n}2-l\right) \varepsilon_{n+1} & \text{ if } k= 0 
                 \end{array}\right\}\\&+ \sum_{i=2}^{m+1}(n/2-i+2)\varepsilon_i + \sum_{i=m+2}^{n-l}(n/2-i+1)\varepsilon_i + \sum_{i=n-l+1}^{n}(n/2-i)\varepsilon_i 
\end{align*}
\end{enumerate}
\begin{itemize}
    \item[\underline{$k=0$}:] 
    \begin{enumerate}
    \setcounter{enumi}{3}
        \item Here the regular character of $\Hi_{0}^{l,m}$ is $2\gamma_{m,l+1}$. So the corresponding trivial infinitesimal character is $\gamma_{m,l+1}$.
        \item Suppose $q+r<n$. The representation $\pi_{m,l+1}$ decomposes in 4 subquotients, namely the $\overline{\pi}_{i,j}$ with $i,j \in\{(m,l+1);(m+1,l+1);(m,l+2);(m+1,l+2)\}$ ($\overline{\pi}_{m+1,l+2}$ is replaced by $\pi_{m+1}$ if $m+l=n-2$). 
        Let $\mu$ and $w$ be as in the algorithm.
        Recall that here $(l,m)\in \{(q,r);(q-1,r);(q,r-1);(q-1,r-1)\}$. The following table gives the correspondence between each trivial infinitesimal character $\gamma_{\cdot,\cdot}$ and the $M$-type which induces the principal series representation with infinitesimal character $w\cdot\mu+ \gamma_{\cdot,\cdot}$, for $(l,m) = (q-1,r)$. The other cases are similar. 
        
        \vspace{5mm}
        \begin{center}
            \begin{tabular}{|Sc|Sc|Sc|}
            \hline
            $\gamma_{r,q}$ & $\sigma_{q-1,r}$ & contains $\tau$ \\
            $\gamma_{r+1,q}$ & $\sigma_{q-1,r+1}$ & does not contain $\tau$\\
            $\gamma_{r,q+1}$ & $\sigma_{q,r}$& contains $\tau$\\
            $\gamma_{r+1,q+1}$ & $\sigma_{q,r+1}$& does not contains $\tau$\\
            \hline
        \end{tabular}
        \end{center}
        and 
\begin{center}
\begin{tikzpicture}[scale = 1]
\draw (-1,-0.5) node  {$\pi_{r,q} =$};
\draw [] (0,0) rectangle (4,1);
\draw [] (0,0) rectangle (4,-1);
\draw [] (1.7,0) to (1.7,-1);
\draw [] (2.3,0) to (2.3,-1);
\draw [] (0,-1) rectangle (4,-2);
\draw (2,0.5) node  {$\overline{\pi}_{r,q}$};
\draw (0.8,-0.5) node  {$\overline{\pi}_{r+1,q}$};
\draw (3.2,-0.5) node  {\color{Rouge}$\overline{\pi}_{r,q+1}$};
\draw (2,-0.5) node  {$\oplus$};
\draw (2,-1.5) node  {$\overline{\pi}_{r+1,q+1}$};
\end{tikzpicture}
\begin{tikzpicture}[scale = 1]
\draw (-1,-0.5) node  {$\pi_{r,q+1} =$};
\draw [] (0,0) rectangle (4,1);
\draw [] (0,0) rectangle (4,-1);
\draw [] (1.7,0) to (1.7,-1);
\draw [] (2.3,0) to (2.3,-1);
\draw [] (0,-1) rectangle (4,-2);
\draw (2,0.5) node  {\color{Rouge}$\overline{\pi}_{r,q+1}$};
\draw (0.8,-0.5) node  {$\overline{\pi}_{r+1,q+1}$};
\draw (3.2,-0.5) node  {$\overline{\pi}_{r,q+2}$};
\draw (2,-0.5) node  {$\oplus$};
\draw (2,-1.5) node  {$\overline{\pi}_{r+1,q+2}$};
\end{tikzpicture}
\end{center}
\begin{center}
\begin{tikzpicture}[scale = 1]
\draw (-1,-0.5) node  {$\pi_{r+1,q} =$};
\draw [] (0,0) rectangle (4,1);
\draw [] (0,0) rectangle (4,-1);
\draw [] (1.7,0) to (1.7,-1);
\draw [] (2.3,0) to (2.3,-1);
\draw [] (0,-1) rectangle (4,-2);
\draw (2,0.5) node  {$\overline{\pi}_{r+1,q}$};
\draw (0.8,-0.5) node  {$\overline{\pi}_{r+2,q}$};
\draw (3.2,-0.5) node  {$\overline{\pi}_{r+1,q+1}$};
\draw (2,-0.5) node  {$\oplus$};
\draw (2,-1.5) node  {$\overline{\pi}_{r+2,q+1}$};
\end{tikzpicture}
\begin{tikzpicture}[scale = 1]
\draw (-1,-0.5) node  {$\pi_{r+1,q+1} =$};
\draw [] (0,0) rectangle (4,1);
\draw [] (0,0) rectangle (4,-1);
\draw [] (1.7,0) to (1.7,-1);
\draw [] (2.3,0) to (2.3,-1);
\draw [] (0,-1) rectangle (4,-2);
\draw (2,0.5) node  {$\overline{\pi}_{r+1,q+1}$};
\draw (0.8,-0.5) node  {$\overline{\pi}_{r+2,q+1}$};
\draw (3.2,-0.5) node  {$\overline{\pi}_{r+1,q+2}$};
\draw (2,-0.5) node  {$\oplus$};
\draw (2,-1.5) node  {$\overline{\pi}_{r+2,q+2}$};
\end{tikzpicture}
\end{center}
\item The only component which is in the two first decompositions but which is not in the two second is $\overline{\pi}_{r,q+1}$. 
And we know that this corresponds to the principal series induced by $\sigma_{q,r}$ and we can compute the $\Aa_\C^*$-part comparing the characters. 
\end{enumerate}
    If $q+r = n$, the unique subquotient which occurs in any principal series written above is the discrete series $\pi_r$. 
    \item[\underline{$k\ne0$}:] 
    \begin{enumerate}
    \setcounter{enumi}{3}
        \item Here the regular character of $\Hi_{0}^{l,m}$ correspond to the trivial infinitesimal character is $\gamma_{0,1}$.
        \item One can check that there is no other candidates which can correspond to the infinitesimal character of $\Hi_{k}^{l,m}$. This means that $\Hi_{k}^{l,m}$ is the only principal series representation in which $\E_k$ can be embedded in. Thus $\E_k$ has to be the irreducible (finite dimensional) quotient of $\Hi_{k}^{l,m}$. So $\E_k$ is $\Phi(\overline{\pi}_{0,1})$ and is finite dimensional.  
    \end{enumerate}
\end{itemize}
\begin{enumerate} \setcounter{enumi}{5}
    \item The Gelfand-Kirillov dimensions are given in Theorem 4. 
\end{enumerate}
\end{proof}

\subsection{Quaternionic case : $G = \Sp(n,1)$, $n\geq 2$}\label{section pforms Sp}

Recall that $K = \Sp(n)\times\Sp(1)$ and \newline$M=\Sp(1)\times \Sp(n-1)\times\Sp(1)$. 

\subsubsection{Decomposition of the representations}

Here the decomposition of the representation is quite long, but very well described in \cite{PedonpformsH}. We will not recall all the facts here, but just describe the example of the $p$-forms when $p=2$, corresponding to $\tau=\tau_2$ and another case with multiplicity 2, namely $\tau=\tau_{1,4,4}$, which will be defined below. 

The decomposition of $\tau_2$ on $K$ is given in \cite[Proposition 4.11]{PedonpformsH}:
\begin{equation}
    \tau_2 = \tau_{0,2,2} \oplus \tau_{0,0,2} \oplus \tau_{1,0,0}~,
\end{equation}
where $\tau_{r,s,t}$ is the $K$-type of highest weight
\begin{equation}
    \mu_{\tau_{r,s,t}} = \sum_{j=1}^{r} 2\varepsilon_j + \sum_{j=r+1}^{r+s} \varepsilon_j + t \varepsilon_{n+1}
\end{equation}
for $r,s,t \in \Z_+$ and $r+s \leq n$. 

For $a,b \in \Z_+$ with $a+b\leq n$ and $2c \in \Z_+$, let $\sigma_{a,b,c}$ be the $M$-type corresponding to the highest weight
\begin{equation}
    \mu_{\sigma_{a,b,c}} =
        c(\varepsilon_1 + \varepsilon_{n+1}) + \sum_{j=2}^{a+1} 2\varepsilon_j + \sum_{j=a+2}^{a+b+1} \varepsilon_j
\end{equation}
Every sum of the form $\sum_{j = l-1}^{l}$ has to be read as $0$. The reader has to be careful, because this highest weight is not the same as the highest weight $\mu_{a,b,c}$ used in \cite[Theorem 5.2]{PedonpformsH} for $\sigma_{a,b,c}$. A simple way to get one from the other is given by the relation
\begin{equation}
  \mu_{a,b,c} = \left\{\begin{array}{ll}
  \mu_{\sigma_{a-1,b,c}}& $if $ a>0,\\\mu_{\sigma_{a,b-1,c}}& $if $ a=0.
  \end{array}\right.
\end{equation}

\begin{Lemma}\label{Lemma decomposition tau2}
The decomposition of the $K$-types of $\tau_2$ and $\tau_{1,4,4}$ over $M$ is given by the following equations
\begin{align}
    \tau_{0,2,2}|_M &= \sigma_{0,2,1} \oplus \sigma_{0,1,1/2} \oplus \sigma_{0,1,3/2} \oplus \sigma_{0,0,1}~,\\
    \tau_{0,0,2}|_M &= \sigma_{0,0,1}~,\\
    \tau_{1,0,0}|_M &= \sigma_{1,0,0} \oplus \sigma_{0,1,1/2} \oplus \sigma_{0,0,1}~,\\
    \tau_{1,4,4}|_M &= \sigma_{1,4,2} \oplus \sigma_{1,3,3/2} \oplus \sigma_{1,3,5/2} \oplus \sigma_{1,2,2}\oplus \sigma_{0,5,3/2} \oplus \sigma_{0,5,5/2}\oplus \sigma_{0,4,1}\oplus \underline{2}\sigma_{0,4,2}\\\notag&\oplus \sigma_{0,4,3}\oplus \sigma_{0,3,3/2} \oplus \sigma_{0,3,5/2}~.
\end{align}
where 
$\sigma_{1,4,2},\sigma_{0,5,3/2}, \sigma_{0,5,5/2}$ will be deleted if $n=5$. The representation $\tau_{1,4,4}$ does not occur if $n<5$. 
\end{Lemma}
The proof is a direct computation using \cite[Theorem 5.2]{PedonpformsH}. 
The order of each decomposition does not matter and follows the list given in the Theorem we used. As written above, we chose $\tau_{1,4,4}$ as a concrete example where $\tau$ occurs  with a multiplicity bigger than 1. We could not have this in the real and complex cases, because these are multiplicity free cases. \textbf{We suppose for this case that $n>5$.}

\subsubsection{Poles of the Plancherel density}
Let $p_{a,b,c}$ be the Plancherel measure of the representation $\sigma_{a,b,c}$. 

\begin{Lemma}\label{zeros Plancherel Measure tau2}
The poles of $p_{a,b,c}$ for the representations $\sigma_{a,b,c}\in \hat{M}(\tau_2)$ are listed in the table below. 

\vspace{5mm}
\begin{center}
 \begin{tabular}{|Sc|Sc|} 
 \hline
 $\sigma\in \hat{M}(\tau_2)$ & Pole $\lambda^{a,b,c}_k$, $k\in \Z_+$\\
 \hline
  $\sigma_{0,0,1}$& $\pm i(n+\frac32+k)$\\
 \hline
$\sigma_{0,1,c}$ , with $2c =1,3$& $\begin{array}{cc}
    \pm i(n+c+\frac12+k) & \text{if } k\ne 0 \\
    \pm i(n-\frac12+c) & \text{if } k=0
\end{array}$\\ 
 \hline
 $\sigma_{1,0,0}$&$\begin{array}{cc}
    \pm i(n+\frac32+k) & \text{if } k\ne 0 \\
    \pm i(n-\frac12) & \text{if } k= 0
\end{array}$\\ 
  \hline
$\sigma_{0,2,1}$ & $\begin{array}{cc}
    \pm i(n+\frac32+k) & \text{if } k\ne 0 \\
    \pm i(n-\frac12) & \text{if } k= 0
\end{array}$\\ 
 \hline
\end{tabular}
\end{center}
\end{Lemma}

\begin{Lemma}\label{zeros Plancherel Measure tau114}
The poles of $p_{a,b,c}$ for the representations $\sigma_{a,b,c}\in \hat{M}(\tau_{1,4,4})$ are listed in the table below. 

\vspace{5mm}
\begin{center}
 \begin{tabular}{|Sc|Sc|} 
 \hline
 $\sigma\in \hat{M}(\tau_{1,1,4})$ & Pole $\lambda^{a,b,c}_k$, $k\in \Z_+$\\
 \hline
  $\sigma_{1,4,2}$& $\begin{array}{cc}
    \pm i(n+\frac52+k) & \text{if } k\ne 0,1 \\
    \pm i\left(n-\cfrac52\right) & \text{if } k= 0\\
    \pm i\left(n+\cfrac52\right) & \text{if } k= 1
\end{array}$\\
 \hline
$\sigma_{1,3,c}$ , with $2c =3,5$& $\begin{array}{cc}
    \pm i(n+c+\frac12+k) & \text{if } k\ne 0,1\\
    \pm i\left(n-\cfrac72+c\right) & \text{if } k= 0\\
    \pm i\left(n+\cfrac12+c\right) & \text{if } k= 1
\end{array}$\\ 
 \hline
 $\sigma_{1,2,2}$& $\begin{array}{cc}
    \pm i(n+\frac52+k) & \text{if } k\ne 0,1 \\
    \pm i\left(n-\cfrac12\right) & \text{if } k= 0\\
    \pm i\left(n+\cfrac52\right) & \text{if } k= 1
\end{array}$\\ 
  \hline
$\sigma_{0,5,c}$, with $2c =3,5$ & $\begin{array}{cc}
    \pm i(n+c+\frac12+k) & \text{if } k\ne 0 \\
    \pm i(n-2) & \text{if } k= 0 \text{ and } c=5/2
\end{array}$\\ 
 \hline
$\sigma_{0,4,c}$, with $c =1,2,3$ & $\begin{array}{cc}
    \pm i(n+c+\frac12+k) & \text{if } k\ne 0 \\
    \pm i(n-\frac72+c) & \text{if } k=0 \text{ and } c=2,3
\end{array}$\\
 \hline
$\sigma_{0,3,c}$, with $2c =3,5$ & $\begin{array}{cc}
    \pm i(n+\frac12+k+c) & \text{if } k\ne 0 \\
    \pm i(n+c-\frac52) & \text{if } k= 0
\end{array}$\\
 \hline
\end{tabular}
\end{center}
\end{Lemma}

\begin{proof}
Consider the arithmetic sequence with a common difference of 1 : 
\begin{equation*}
    \pm i \left(-c-\frac32 +n+{\tiny\left\{\begin{array}{c}
         0  \\ 1\\2
    \end{array}\right\}}\right)…\pm i \left(-c+\frac12 -n-{\tiny\left\{\begin{array}{c}
         0  \\ 1\\2
    \end{array}\right\}}\right)~,
\end{equation*}
where we choose $\left\{\begin{array}{ll}
    0 & \text{if } a=b=0 \\
    1 & \text{if } a=0 \text{ and } b\ne0\\
    2 & \text{if } a\ne 0 \\
\end{array}\right.$.
Moreover, in the case when we choose 1, we remove the sequence by the 4 numbers $\pm i\left(-c-\frac12+n-a-b\right)$ and $\pm i\left(-c-\frac12-n+a+b\right)$ (here $a$ is 0). In the case when we choose 2, we also remove, in addition to the previous 4, the 4 numbers $\pm i\left(-c+\frac12+n-a\right)$ and $\pm i\left(-c-\frac32-n+a\right)$.
Using \cite[Theorem 3.1]{Mia} or \cite[Proposition 3.1]{ROBY2021}, one can prove that the zeros of the polynomial part of $p_{a,b,c}$ are given by these removed sequences. The complement of these zeros in $i\Z$ or in $i\left(\Z+\frac12\right)$, depending where the zeros are, is the set of the poles of $p_{a,b,c}$.  
\end{proof}

In the following, we make the proof for $\sigma_{0,4,2}$ which is the only component with multiplicity 2. The other cases are similar to those considered in the real and the complex cases.

\subsubsection{Proof for $\sigma_{0,4,2}$}
We now have all the ingredients to study the residue representation $\E^{0,4,2}_k :=\E^{\sigma_{0,4,2}}_k$ described in section \ref{section residue repr}.  
The residue representations in \eqref{eq residue representation} become 
\begin{equation}
\E_k^{0,4,2} := \{\varphi_{1,4,4}^{0,4,2,k} \ast f ~|~ f \in  C_c^\infty(G,\tau_{p,q})\}~,
\end{equation}
for $k \in \Z_+$. To simplify notations, we set $\varphi_{1,4,4}^{0,4,2,k}:=\varphi_{\tau_{1,4,4}}^{\sigma_{0,4,2},\lambda^{0,4,2}_k\alpha}$  for the spherical function and we will denote the principal series representations $\Ind_{MAN}^G(\sigma_{0,4,2}\otimes e^{i\lambda_k^{0,4,2}}\otimes 1)$, in which $\E_k^{0,4,2}$ is embedded, by $\Hi_k^{0,4,2}:=\Hi^{\sigma_{0,4,2}}_{\lambda_k^{0,4,2}}~.$

\begin{enumerate}
    \item and (2) were already proved before.
    \setcounter{enumi}{2}
    \item The infinitesimal character of $\Hi_k^{0,4,2}$ is given by 
\begin{align*}
    \gamma_k^{0,4,2} &= i\lambda^{0,4,2}_k(\varepsilon_1+\varepsilon_{2}) + \mu_{\sigma_{0,4,2}} + \rho_m \\
                 &= \left\{\begin{array}{cc}
                     (n+k+5)\varepsilon_1 + (n+k)\varepsilon_{2} & \text{ if } k\in \Z_+^\times \\
                      \left(n+1\right) \varepsilon_1 + \left(n-4\right) \varepsilon_2 & \text{ if } k= 0 
                 \end{array}\right\}+ \sum_{i=3}^{6}(n-i+3)\varepsilon_i + \sum_{i=7}^{n+1}(n-i+2)\varepsilon_i
\end{align*}
\end{enumerate}

It is immediate to see that for $k\geq 2$, $\gamma_k^{0,4,2}$ corresponds to $\gamma_{0,1}$ by Theorem 5 and cannot be send by an element of $W(\g_\C,\h_\C)$, to a regular character induced by an other representation of $\hat{M}(\tau_{1,4,4})$.  

\noindent\underline{\bf k=1}: 
\begin{enumerate}
\setcounter{enumi}{3}
    \item The functors send $\gamma_k^{0,4,2}$ to the trivial character $\gamma_{0,1}$.
    \item  $\gamma_k^{0,4,2}$ has just one candidate: $\Ind_{MAN}^G(\sigma_{1,3,5/2}\otimes e^{(n+3)\alpha}\otimes 1)$ with infinitesimal charcter $\gamma_{(n+3)\alpha}^{1,3,5/2}$.
Moreover, we have :
\begin{center}
\begin{tikzpicture}[scale = 0.7]
\draw (-1.4,-0.5) node  {$\pi_{0,1} =$};
\draw [] (0,0) rectangle (3,1);
\draw [] (0,0) rectangle (3,-1);
\draw [] (0,-1) rectangle (3,-2);
\draw [] (0.2,0) to (0.2,-1);
\draw [] (0.4,-2) to (0.4,-1);
\draw (1.5,0.5) node  {$\overline{\pi}_{0,1}$};
\draw (1.7,-0.5) node  {\color{Rouge}$\overline{\pi}_{0,2}$};
\draw (1.9,-1.5) node  {$\overline{\pi}_{2,3}$};
\end{tikzpicture}\hspace{2cm}
    \begin{tikzpicture}[scale = 0.7]
\draw (-1.5,-0.5) node  {$\pi_{0,2} =$};
\draw [] (0,0) rectangle (9,1);
\draw [] (0,0) rectangle (9,-1);
\draw [] (3.4,0) to (3.4,-1);
\draw [] (4,0) to (4,-1);
\draw [] (6.7,0) to (6.7,-1);
\draw [] (6.1,0) to (6.1,-1);
\draw [] (0,-1) rectangle (9,-2);
\draw (4.5,0.5) node  {\color{Rouge}$\overline{\pi}_{0,2}$};
\draw (1.8,-0.5) node  {$\overline{\pi}_{2,3}$};
\draw (5.1,-0.5) node  {$\overline{\pi}_{0,3}$};
\draw (7.8,-0.5) node  {$\overline{\pi}_{1,2}$};
\draw (3.7,-0.5) node  {$\oplus$};
\draw (6.4,-0.5) node  {$\oplus$};
\draw (4.5,-1.5) node  {$\overline{\pi}_{1,3}$};
\end{tikzpicture}
\end{center}

So $\E_1^{0,4,2}$ corresponds to $\overline{\pi}_{0,2}$. It cannot be $\overline{\pi}_{2,3}$, because this one appears in another principal series (namely $\pi_{2,3}$), which does not contain $\E_1^{0,4,2}$. 
Here the two $K$-types $\tau_{1,4,4}$ appear in the same component, and $\E_1^{0,4,2}$ has Langlands parameters $(MA,\sigma_{1,3,5/2},(n+3)\alpha)$. 
\end{enumerate}

\noindent\underline{\bf k=0}: 
\begin{enumerate}
\setcounter{enumi}{3}
    \item The functors send $\gamma_k^{0,4,2}$ to the trivial character $\gamma_{0,5}$.
    \item Here there are four different $\sigma$ which are potential candidates. The following table gives all the information about the other principal series with the ‘‘same’’ infinitesimal characters and induced by a $M$-type in $\hat{M}(\tau_{1,4,4})$. 

\vspace{2mm}
\hspace{-20mm}
\begin{tabular}{|Sc|Sc|Sc|Sc|Sc|}
\hline
$\sigma_{0,4,2}$& $\sigma_{1,4,2}$&$\sigma_{1,3,3/2}$&$\sigma_{0,5,5/2}$&$\sigma_{0,3,3/2}$\\
\hline
$\gamma_{0,5}$& $\gamma_{1,6}$& $\gamma_{1,5}$& $\gamma_{0,6}$& $\gamma_{0,4}$\\
\hline
\begin{tikzpicture}[scale = 0.5]
\draw (-1,-0.5) node  {$\pi_{0,5} =$};
\draw [] (0,0) rectangle (4,1);
\draw [] (0,0) rectangle (4,-1);
\draw [] (1.7,0) to (1.7,-1);
\draw [] (2.3,0) to (2.3,-1);
\draw [] (0,-1) rectangle (4,-2);
\draw (2,0.5) node  {\color{Bleu}$\overline{\pi}_{0,5}$};
\draw (0.8,-0.5) node  {$\overline{\pi}_{1,5}$};
\draw (3.2,-0.5) node  {$\overline{\pi}_{0,6}$};
\draw (2,-0.5) node  {$\oplus$};
\draw (2,-1.5) node  {\color{Rouge}$\overline{\pi}_{1,6}$};
\end{tikzpicture}
& \begin{tikzpicture}[scale = 0.5]
\draw (-1,-0.5) node  {$\pi_{1,6} =$};
\draw [] (0,0) rectangle (4,1);
\draw [] (0,0) rectangle (4,-1);
\draw [] (1.7,0) to (1.7,-1);
\draw [] (2.3,0) to (2.3,-1);
\draw [] (0,-1) rectangle (4,-2);
\draw (2,0.5) node  {\color{Rouge}$\overline{\pi}_{1,6}$};
\draw (0.8,-0.5) node  {$\overline{\pi}_{2,6}$};
\draw (3.2,-0.5) node  {$\overline{\pi}_{1,7}$};
\draw (2,-0.5) node  {$\oplus$};
\draw (2,-1.5) node  {$\overline{\pi}_{2,7}$};
\end{tikzpicture}
&\begin{tikzpicture}[scale = 0.5]
\draw (-1,-0.5) node  {$\pi_{1,5} =$};
\draw [] (0,0) rectangle (4,1);
\draw [] (0,0) rectangle (4,-1);
\draw [] (1.7,0) to (1.7,-1);
\draw [] (2.3,0) to (2.3,-1);
\draw [] (0,-1) rectangle (4,-2);
\draw (2,0.5) node  {$\overline{\pi}_{1,5}$};
\draw (0.8,-0.5) node  {$\overline{\pi}_{2,5}$};
\draw (3.2,-0.5) node  {\color{Rouge}$\overline{\pi}_{1,6}$};
\draw (2,-0.5) node  {$\oplus$};
\draw (2,-1.5) node  {$\overline{\pi}_{2,6}$};
\end{tikzpicture}
&\begin{tikzpicture}[scale = 0.5]
\draw (-1,-0.5) node  {$\pi_{0,6} =$};
\draw [] (0,0) rectangle (4,1);
\draw [] (0,0) rectangle (4,-1);
\draw [] (1.7,0) to (1.7,-1);
\draw [] (2.3,0) to (2.3,-1);
\draw [] (0,-1) rectangle (4,-2);
\draw (2,0.5) node  {$\overline{\pi}_{0,6}$};
\draw (0.8,-0.5) node  {\color{Rouge}$\overline{\pi}_{1,6}$};
\draw (3.2,-0.5) node  {$\overline{\pi}_{0,7}$};
\draw (2,-0.5) node  {$\oplus$};
\draw (2,-1.5) node  {$\overline{\pi}_{1,7}$};
\end{tikzpicture}
&\begin{tikzpicture}[scale = 0.5]
\draw (-1,-0.5) node  {$\pi_{0,4} =$};
\draw [] (0,0) rectangle (4,1);
\draw [] (0,0) rectangle (4,-1);
\draw [] (1.7,0) to (1.7,-1);
\draw [] (2.3,0) to (2.3,-1);
\draw [] (0,-1) rectangle (4,-2);
\draw (2,0.5) node  {$\overline{\pi}_{0,4}$};
\draw (0.8,-0.5) node  {$\overline{\pi}_{1,4}$};
\draw (3.2,-0.5) node  {\color{Bleu}$\overline{\pi}_{0,5}$};
\draw (2,-0.5) node  {$\oplus$};
\draw (2,-1.5) node  {$\overline{\pi}_{1,5}$};
\end{tikzpicture}\\
\hline
\end{tabular}
\vspace{2mm}

The two $K$-types $\tau_{1,4,4}$ are embedded in two different components, namely $\overline{\pi}_{0,5}$ and $\overline{\pi}_{1,6}$. The choice is made because these are the only maximal subquotients which appear in $\Hi^{0,4,2}_0$. The component $\overline{\pi}_{1,6}$ is obligatory because it is the only one which appears in $\pi_{1,6}$. Then
it cannot be $\overline{\pi}_{1,5}$ (or $\overline{\pi}_{0,6}$), because the multiplicity of $\tau_{1,4,4}$ is one in $\pi_{1,5}$ (or $\pi_{0,6}$). 
\end{enumerate}

\subsubsection{The residue representations}

We denote then the residue representation by $\E^{a,b,c}_k$. We recall that, for $\tau$ a $K$-type of $\tau_2$, this is the left action of $G$ on the space 
\begin{equation}
    \E_k^{a,b,c} := \{\varphi_{\tau}^{\sigma_{a,b,c},\lambda_k^{a,b,c}} \ast f ~|~ f \in  C_c^\infty(G,\tau)\}~.
\end{equation}

\begin{Prop}[For $\tau_{0,2,2}$]\label{Prop res repr pforms - tau022}
The residue representation $\E^{a,b,c}_k$ is always irreducible and:

\begin{itemize}
    \item \underline{if $k\ne0$}: finite dimensional with Langlands parameters $(MA,\sigma_{a,b,c},\pm i\lambda_k^{a,b,c}\alpha)$,
    \item \underline{if $k = 0$ and $(a,b,c)\in\{(0,0,1); ~(0,1,3/2)\}$}: infinite dimensional with Gelfand-Kirillov dimension $2n+1$ and Langlands parameters $(MA,\sigma_{0,1,3/2},\pm i\lambda_0^{0,1,3/2}\alpha)$,
    \item \underline{if $k = 0$ and $(a,b,c)\in\{(0,2,1); ~(0,1,1/2)\}$}: infinite dimensional with Gelfand-Kirillov dimension $2n+1$ and Langlands parameters $(MA,\sigma_{0,2,1},\pm i\lambda_0^{0,2,1}\alpha)$.
\end{itemize}
\end{Prop}

\begin{Prop}[For $\tau_{0,0,2}$]\label{Prop res repr pforms - tau002}
The residue representation $\E^{a,b,c}_k$ is always irreducible and finite dimensional with Langlands parameters $(MA,\sigma_{0,0,1},\pm i\lambda_k^{0,0,1}\alpha)$.
\end{Prop}

\begin{Prop}[For $\tau_{1,0,0}$]\label{Prop res repr pforms - tau100}
The residue representation $\E^{a,b,c}_k$ is always irreducible and:
\begin{itemize}
    \item \underline{if $k\ne0$ or $(a,b,c) = (0,0,1)$}: finite dimensional with Langlands parameters $(MA,\sigma_{a,b,c},\pm i\lambda_k^{a,b,c}\alpha)$,
    \item \underline{if $k = 0$ and $(a,b,c)\ne(0,0,1)$}: infinite dimensional with Gelfand-Kirillov dimension $2n+1$ and Langlands parameters $(MA,\sigma_{1,0,0},\pm i\lambda_0^{1,0,0}\alpha)$.
\end{itemize}
\end{Prop}

\begin{Prop}[For $\tau_{1,4,4}$, $n>5$]\label{Prop res repr pforms - tau144}
If $(a,b,c,k)\ne (0,4,2,0)$, the residue representation $\E^{a,b,c}_k$ is irreducible. If $\E^{a,b,c}_k$ is not listed below, it is finite dimensional with Langlands parameter $(MA,\sigma_{a,b,c},\lambda_k^{a,b,c})$. 
\begin{itemize}
    \item \underline{if $k = 0$ and $(a,b,c)\in\{(1,4,2),(1,3,3/2),(0,5,5/2)\}$}: infinite dimensional with Gelfand-Kirillov dimension $4n-2$ and Langlands parameters $(MA,\sigma_{1,4,2},\lambda_0^{1,4,2})$,
    \item \underline{if $k = 1$ and $(a,b,c)\in\{(1,4,2),(0,5,3/2)\}$}: infinite dimensional with Gelfand-Kirillov dimension $2n+1$ and Langlands parameters $(MA,\sigma_{1,4,2},\lambda_1^{1,4,2})$,
    \item \underline{if $k = 1$ and $(a,b,c)\in\{(1,3,3/2),(0,4,1)\}$}: infinite dimensional with Gelfand-Kirillov dimension $2n+1$ and Langlands parameters $(MA,\sigma_{1,3,3/2},\lambda_1^{1,3,3/2})$,
    \item \underline{if $k = 0$ and $(a,b,c)\in\{(1,3,5/2),(1,2,2)\}$}: infinite dimensional with Gelfand-Kirillov dimension $2n+1$ and Langlands parameters $(MA,\sigma_{1,3,5/2},\lambda_0^{1,3,5/2})$,
    \item \underline{if $k = 1$ and $(a,b,c)\in\{(1,3,5/2),(0,4,2)\}$}: infinite dimensional with Gelfand-Kirillov dimension $2n+1$ and Langlands parameters $(MA,\sigma_{1,3,5/2},\lambda_1^{1,3,5/2})$,
    \item \underline{if $k = 1$ and $(a,b,c)\in\{(1,2,2),(0,3,3/2)\}$}: infinite dimensional with Gelfand-Kirillov dimension $2n+1$ and Langlands parameters $(MA,\sigma_{1,2,2},\lambda_1^{1,2,2})$,
    \item \underline{if $k = 0$ and $(a,b,c)\in\{(0,4,3),(0,3,5/2)\}$}: infinite dimensional with Gelfand-Kirillov dimension $2n+1$ and Langlands parameters $(MA,\sigma_{0,4,3},\lambda_1^{0,4,3})$,
    \item \underline{if $k = 0$ and $(a,b,c)=(0,3,3/2)$}: infinite dimensional with Gelfand-Kirillov dimension $2n+1$ and Langlands parameters $(MA,\sigma_{0,4,2},\lambda_1^{0,4,2})$.
\end{itemize}
The representation $\E^{0,4,2}_0$ is the sum of two representations. The one of Langlands parameter $(MA,\sigma_{0,4,2},\lambda_0^{0,4,2})$ and Gelfand-Kirillov dimension $2n+1$ and the other of Langlands parameter $(MA,\sigma_{1,4,2},\lambda_0^{1,4,2})$ and Gelfand-Kirillov dimension $4n-2$. 
\end{Prop}
\bibliographystyle{alpha}
\bibliography{Biblio}
\end{document}